\documentclass[paper=a4, fontsize=11pt]{scrartcl}

\usepackage[T1]{fontenc}
\usepackage[utf8]{inputenc}
\usepackage[english]{babel}
\usepackage{amsmath,amsfonts,amsthm,amssymb}
\usepackage[left=1.9cm,right=1.9cm,top=2.85cm,bottom=2.85cm]{geometry}

\usepackage[Algorithm]{algorithm} 
\usepackage{algorithmic}
\usepackage{graphicx}
\usepackage{setspace}

\usepackage{xcolor}

\newtheorem{theorem}{Theorem}[section]
\newtheorem{lemma}[theorem]{Lemma}

\newtheorem{definition}[theorem]{Definition}

\newtheorem{assumption}{Assumption}

\usepackage{hyperref}

\numberwithin{equation}{section}
\numberwithin{algorithm}{section}
\numberwithin{figure}{section}
\numberwithin{assumption}{section}

\usepackage[german]{fancyref}
\usepackage[fixlanguage]{babelbib}

\DeclareMathAlphabet{\mathcall}{OMS}{cmsy}{m}{n}

\newcommand{\cl}[1]{\operatorname{cl}\!\left(#1\right)}
\newcommand{\deriv}[1]{\frac{\mathrm{d}}{\mathrm{d}#1}}
\newcommand{\xCinfty}{C^{\infty}}
\newcommand{\xCzero}{C^{0}}
\newcommand{\xCone}{C^{1}}
\providecommand{\keywords}[1]{\textbf{\textit{Key words---}} #1}

\providecommand{\MSC}[1]{\textbf{\textit{MSC---}} #1}
\begin{document}
\selectlanguage{\english}
  \title{On the Convergence of the Policy Iteration for Infinite-Horizon Nonlinear Optimal Control Problems}
  \author{Tobias Ehring \thanks{Institute of Applied Analysis and Numerical Simulation, University of Stuttgart, Pfaffenwaldring 57, Stuttgart 70569, Baden-W\"urttemberg, Germany. E-mail: ehringts@mathematik.uni-stuttgart.de} \and  Behzad Azmi\thanks{Department of Mathematics and Statistics, University of Konstanz, Universit\"atsstra\ss e 10, Konstanz 78457, Baden-W\"urttemberg, Germany. E-mail: behzad.azmi@uni-konstanz.de} \and Bernard Haasdonk\thanks{Institute of Applied Analysis and Numerical Simulation, University of Stuttgart, Pfaffenwaldring 57, Stuttgart 70569, Baden-W\"urttemberg, Germany. E-mail: haasdonk@mathematik.uni-stuttgart.de}}
  \date{\today}
  \maketitle
\begin{abstract} \noindent Policy iteration (PI) is a widely used algorithm for synthesizing optimal feedback control policies across many engineering and scientific applications. When PI is deployed on infinite-horizon, nonlinear, autonomous optimal-control problems, however, a number of significant theoretical challenges emerge -- particularly when the computational state space is restricted to a bounded domain. In this paper, we investigate these challenges and  show that the viability of PI in this setting hinges on the existence, uniqueness, and regularity of solutions to the Generalized Hamilton–Jacobi–Bellman (GHJB) equation solved at each iteration. To ensure a well-posed iterative scheme, the GHJB solution must possess sufficient smoothness, and the domain on which the GHJB equation is solved must remain forward-invariant under the closed-loop dynamics induced by the current policy. Although fundamental to the method’s convergence, previous studies have largely overlooked these aspects. This paper closes that gap by introducing a constructive procedure that guarantees forward invariance of the computational domain throughout the entire PI sequence and by establishing sufficient conditions under which a suitably regular GHJB solution exists at every iteration. Numerical results are presented for a grid-based implementation of PI to support the theoretical findings. \end{abstract}


\MSC{49L20,  49N35,  49J15,  49L12}

\section{Introduction}
\noindent The synthesis of optimal feedback policies is a fundamental strategy for achieving robust, real-time control of dynamical systems in the presence of uncertainties, disturbances, and modeling errors. Such policies adapt control actions to the current state, thereby generating trajectories that minimize a predetermined cost. In this work, we consider a class of nonlinear autonomous optimal control problems (OCPs) defined over an infinite time horizon. More precisely, we investigate systems in which the control appears affinely in the dynamics and quadratically in the cost functional. Our objective is to obtain such an optimal feedback policy for  problems of the form
\begin{gather}
\tag{OCP}
v^*(x) =  \inf_{\mathbf{u} \in \mathcal{U}_{\infty}} \int_0^{\infty}  \Bigl[ h( \mathbf{x}(t)) + \left\langle  \mathbf{u}(t), R  \mathbf{u}(t) \right \rangle \Bigr] \mathrm d t \label{eq:MPInfinit}, \\ 
\text{s.t. }\quad \dot{\mathbf{x}}(t) = f(\mathbf{x}(t))+g(\mathbf{x}(t))\mathbf{u}(t) , \quad \forall t \in [0,\infty) \quad \text{ and } \quad \mathbf{x}(0) = x \in \mathbb{R}^N, 
\label{eq:ODEInfinit}
\end{gather}
with the admissible control set defined as 
\[
\mathcal{U}_{\infty} := \Big\{ \mathbf{u}: [0,\infty) \to \mathbb{R}^M \,\big|\, \mathbf{u}\text{ is measurable and essentially bounded} \Big\}.
\]
Here, the function \(v^*(x)\), also known as the optimal value function (VF), represents the optimal cost-to-go from the initial state \(x\). In what follows, time-dependent scalar and vector-valued functions are denoted using bold symbols to distinguish them from time-independent quantities. For a controller from the admissible control set $\mathcal{U}_{\infty}$, the solution to \eqref{eq:ODEInfinit} is initially understood in the Carathéodory sense \cite{Filippov1988}. However, since we subsequently restrict our analysis to at least continuous controllers, solutions to \eqref{eq:ODEInfinit} will ultimately be considered in the classical sense. Sufficient conditions ensuring well-posedness of \eqref{eq:MPInfinit} are provided in \cite{Dmitruk2005}. Those technical assumptions are not reproduced here; instead, wherever needed in the theorems below, we  assume the existence of an optimal control.\\
 Throughout the analysis, we impose the following regularity conditions on the system data. These conditions ensure that the system dynamics is smooth, and that the running cost $h(x) + \left\langle  u, R  u \right \rangle$ is strictly positive except at the origin.
\begin{assumption}\label{as:data}
  Let \(f \in \xCinfty(\mathbb{R}^N,\mathbb{R}^N)\), \(g \in \xCinfty(\mathbb{R}^N,\mathbb{R}^{N \times M})\), \(h \in \xCinfty(\mathbb{R}^N,\mathbb{R})\) and let \(R \in \mathbb{R}^{M\times M}\) be a symmetric positive-definite matrix (i.e.\ \(\left\langle    u,R u\right \rangle  >0\) for all \(u \in \mathbb{R}^M\! \setminus \! \{0\}\)). 
  Moreover, assume \(f(0) = 0\) and that \(h\) is a positive-definite function (i.e.\ \(h(x)>0\) for all \(x \in \mathbb{R}^N\! \setminus \! \{0\}\) and \(h(0)=0\)).
\end{assumption}
\noindent A fundamental concept in dynamic optimization is Bellman’s principle of optimality (see, e.g., \cite[Theorem 1.1]{Kamalapurkar2018}). This principle enables the reformulation of the infinite-horizon problem as a family of finite-horizon problems. For any finite terminal time $T>0$, the problem \eqref{eq:MPInfinit} can equivalently be expressed  as follows
 \begin{gather}
v^*(x) = \inf_{\mathbf{u} \in \mathcal{U}_{T}} \left[\int_0^{T} \Bigl[ h(\mathbf{x}(t)) +\left\langle  \mathbf{u}(t), R \mathbf{u}(t)\right \rangle \Bigr] \mathrm d t +v^*(\mathbf{x}(T))\right],  \label{eq:MPFinit} \\ 
\text{s.t. } \quad \dot{\mathbf{x}}(t) = f(\mathbf{x}(t))+g(\mathbf{x}(t))\mathbf{u}(t), \quad \forall t \in [0,T] \quad \text{ and } \quad \mathbf{x}(0) = x, \label{eq:ODEFinite}
\end{gather}
with $$\mathcal{U}_{T} := \{ \mathbf{u}: [0, T ] \rightarrow  \mathbb{R}^M \, \vert \, \mathbf{u}\text{ measurable and essentially bounded} \}.$$
\noindent This formulation reflects the principle that any tail of an optimal trajectory is itself optimal. Henceforth,   a domain is a non-empty, open, and connected subset of  $\mathbb{R}^N$. Provided that the optimal VF $v^*$ is continuously differentiable on a domain $\Omega \subset \mathbb{R}^N$  containing the origin,  we can derive a partial differential equation (PDE) characterizing $v^*$ (see, e.g., \cite{Bardi1997}) by rearranging \eqref{eq:MPFinit}--\eqref{eq:ODEFinite}, dividing by $T$, and taking the limit as $T \to 0$,
\begin{align}
\label{eq:HBJ}
\min_{u \in \mathbb{R}^M}\left\lbrace \left\langle f(x)+g(x)u,  \nabla v^*(x)\right\rangle + h(x) + \left\langle u , R u \right\rangle \right\rbrace = 0 
\; \text{ for } x \in \Omega, \text{ with } v^*(0) = 0.
\end{align} 
\noindent 
This equation is known as the Hamilton–Jacobi–Bellman (HJB) equation. It involves a convex optimization problem with respect to  $u$, which can be solved exactly using the first-order optimality condition. This yields
\begin{align}\label{eq:HBJ2}
\tag{HJB}
\left\langle f(x),  \nabla v^*(x)\right\rangle  - \frac{1}{4}  \nabla {v^*}^{\top}(x) g^{\top}(x)  R^{-1} g(x) \nabla v^*(x) +h(x) = 0 \; \text{ for } x \in \Omega,  \text{ with } v^*(0) = 0,
\end{align}
where the minimizer is given explicitly by
\begin{align} 
\mathcal{K}(x;\nabla v^*) := u^*(x) = -\frac{1}{2}R^{-1}g(x)^{\top} \nabla v^*(x). \label{eq:feedback}
\end{align}
Inserting this feedback control law into \eqref{eq:ODEInfinit} produces the optimal closed-loop dynamics. Consequently, the design of an optimal feedback controller reduces to computing the associated optimal VF.  However, because of the nonlinearity and the potentially high-dimensional state space of \eqref{eq:HBJ2}, solving it directly for the optimal VF is a challenging task.\\
A widely adopted algorithm to tackle \eqref{eq:HBJ2} is Policy Iteration (PI). Originally introduced by Bellman \cite{Bellman2010} and rigorously analyzed by Howard \cite{Howard1958} in the context of discrete dynamic programming, the method was subsequently translated to continuous-time systems by Leake \cite{Leake1967} and Vaisbord \cite{Vaisbord1963}. Notably, PI can also be interpreted as a Newton-type iterative method applied to the nonlinear equation \eqref{eq:HBJ2}. At each iteration, given a control policy $u$, the algorithm first executes a policy evaluation step by computing the VF $v_u$ as the solution to the linearized HJB equation, commonly referred to as the generalized HJB (GHJB) equation
\begin{equation}\label{eq:GHJB}
\mathrm{GHJB}(v_u,u) := \left\langle f(x)+g(x)u,  \nabla v_u(x)\right\rangle  + h(x) + \left\langle   u(x) ,  R\, u(x) \right \rangle = 0  \; \text{ for } x \in \Omega,   \text{ with } v_u(0)=0.
\end{equation}
The solution $v_u$ indeed provides the performance index for the current control policy $u$, i.e,
\begin{gather}\label{eq:VFu}
v_u(x) =   \int_0^{\infty} \Bigl[ h( \mathbf{x}_u(t;x)) + \left\langle   u( \mathbf{x}_u(t;x)), R  u( \mathbf{x}_u(t;x))\right \rangle \Bigr] \mathrm d t,
\end{gather}
where $\mathbf{x}_u(t;x)$ denotes the trajectory of the closed-loop system   \begin{gather*}
\dot{\mathbf{x}}_u(t;x) = f(\mathbf{x}_u(t;x))+g(\mathbf{x}_u(t;x))u(\mathbf{x}_u(t;x)) \; \text{ and  }\, \mathbf{x}_u(0;x) = x.
\end{gather*} 
Subsequently, the policy is refined via a policy improvement step by updating the control law according to
\[
u^+(x) := -\frac{1}{2} R^{-1} g(x)^{\top} \nabla v_u(x).
\]
This is iteratively repeated until convergence, i.e., until a fixed point is reached at which the corresponding VF satisfies \eqref{eq:HBJ2}. A generic version of the PI algorithm is outlined in Algorithm \ref{algo:PI}.
  \setcounter{algorithm}{1}
\begin{algorithm}[H]
\caption{Generic policy iteration}\label{algo:PI}
\begin{flushleft}
\textbf{Input:}\\[2pt]
\hspace*{1.5em}initial feedback \(u_{0}\);\; convergence tolerance \(\varepsilon>0\)\\[6pt]

\textbf{Initialisation:}\\[2pt]
\hspace*{1.5em}%
\(v_{-1}\equiv 0,\;
e_{0}:=\varepsilon+1,\;
s:= 0\) \\[6pt]

\textbf{Main loop:}\\[2pt]
\textbf{while} \(e_{s}>\varepsilon\)\ \textbf{do}\\
\hspace*{1.5em}%
\textbf{(1)}\; solve   
$
\mathrm{GHJB}\bigl(v_{s},u_{s}\bigr)=0 $ and $v_{s}(0)=0$ for the value function $v_s$;
\\[4pt]
\hspace*{1.5em}%
\textbf{(2)}\; update the feedback law
$
u_{s+1}(x)\;:=\;-\tfrac12\,R^{-1}g(x)^{\!\top}\nabla v_{s}(x);
$\\[4pt]
\hspace*{1.5em}%
\textbf{(3)}\; compute the residual  
\(e_{s+1}\;:=\;\lVert v_{s}-v_{s-1}\rVert;\)\\[4pt]
\hspace*{1.5em}%
\textbf{(4)}\; increment the counter \(s\;:=\;s+1\).\\
\textbf{end while}\\[6pt]

\textbf{Output:}\\[2pt]
\hspace*{1.5em}%
approximate optimal value function \(v_{\,s-1}\) and feedback \(u_{s}\).
\end{flushleft}
\end{algorithm}
\noindent   A critical step in the PI algorithm involves solving the  GHJB equation.  This step raises two important issues that have not been fully addressed in the existing literature and clarifying these issues is the main theoretical contribution of this work. Specifically, at each iteration of the PI algorithm, it is necessary to determine: 
\begin{itemize}
    \item[1.] On which computational domain \(\Omega \subset \mathbb{R}^N\) must the GHJB be solved?
    \item[2.] Under what conditions on the current controller \(u\) does the GHJB admit a continuously differentiable unique solution $v_s$?
\end{itemize}
The first question is particularly important because the updated controller \( u^{+} \), derived from the solution \( v_u \) associated with the current controller \( u \), may generate trajectories that exit the current computational domain \( \Omega \). However, since \( v \) is defined only on \( \Omega \), the application of \( u^{+} \) outside this domain is not feasible. Therefore, it is necessary to identify a new subset \( \Omega^{+} \subset \Omega \) such that the trajectories governed by the updated controller \( u^+ \) remain entirely within \( \Omega^+ \).\\
Most of the existing literature assumes global solvability of  \eqref{eq:HBJ2} over the entire domain  \(\Omega = \mathbb{R}^N\)  (see, for example, \cite{Saridis1979, Saridis1984}). However, for nonlinear infinite-horizon problems, this assumption often does not hold.
For a convergence proof of PI in the case  \(\Omega = \mathbb{R}^N\), we can mention e.g., \cite[Theorem 3.1.4]{Jiang2017}. In particular, for the Linear-Quadratic Regulator (LQR) optimal control problem, the analysis is naturally carried out over the entire state space \(\mathbb{R}^N\). In this setting, the system dynamics are given by
\[
f(x) = Ax,\quad g(x) = B,
\]
for matrices $A \in \mathbb{R}^{N \times N}$, $B\in \mathbb{R}^{N \times M}$ and the running state cost is defined as
\[
h(x) = x^\top Q x,
\]
with \(Q \in \mathbb{R}^{N \times N}\) positive-definite.
Here, all VFs in the PI algorithm reduce to quadratic forms and are therefore fully characterized by a finite set of coefficients, which can be determined by solving a Lyapunov equation derived from the GHJB equation. As a result, Algorithm~\ref{algo:PI} simplifies to the well-known Newton–Kleinman algorithm for computing the solution to the algebraic Riccati equation, originally introduced in \cite{Kleinman1968}. However, in the general nonlinear case, expecting solvability of the GHJB equation in \(\mathbb{R}^N\) is not realistic.\\
 Concerning the second question, under mild assumptions on the controller \(u_s\) -- namely, continuity and admissibility in the sense that \eqref{eq:VFu} is finite, one can guarantee the existence of a continuous solution to the GHJB equation, denoted by \(v_{u_s}\). In general, this solution is only differentiable along the direction of the dynamics, i.e., along the vector field given by
$
f(x) + g(x)u_s(x),
$
which is necessary for the GHJB equation to be well-posed. However, to ensure that Step 2 in Algorithm \(\ref{algo:PI}\) is well-defined, overall differentiability of \(v_s\) is required. We address this issue by imposing the regularity assumption stated in Assumption \(\ref{as:data}\).\\
In literature, the first rigorous analyses for continuous-time optimal control problems were presented in \cite{Saridis1979} and \cite{Saridis1984}. In these works, convergence theorems were established under the assumption that the VF is differentiable at every iteration and, consequently, satisfies the GHJB equation. Therefore, if a control update yields a VF that is not differentiable, the PI algorithm may fail. In subsequent publications \cite{Beard1995,Beard1997,Beard1998}  this result was inaccurately cited. Specifically, these works inferred the differentiability of the VF as a consequence of the admissibility of the controller, whereas in the original studies, differentiability was explicitly assumed a priori. This misinterpretation has been propagated in later research \cite{Dolgov2021,Kalise2020,Vrabie2009,Vrabie2009a,Kalise2018}, leading to overly simplified and incomplete convergence proofs for PI. This inaccuracy may partly account for the limited amount of rigorous convergence analysis in the literature, as the issue was prematurely regarded as resolved. Nevertheless, several more recent studies \cite{Vamvoudakis2009,Vamvoudakis2010,Kundu2021,Jiang2017,Kamalapurkar2018} continue to follow the original formulation presented in \cite{Saridis1979}, maintaining the assumption of VF differentiability. However, these works also place a greater emphasis on the numerical aspects of the PI method.\\
Generally, approximating the optimal VF using PI has gained significant interest in recent years. Numerous numerical schemes have been introduced, each mainly differing in the method used to approximate the solution to the GHJB equation. These approaches include:
\begin{itemize}
    \item \textbf{Low-Rank Approximation Techniques:} In high-dimensional state spaces, low-rank approximation methods \cite{Eigel2022,Dolgov2021} have been developed to efficiently represent the VFs.
    \item \textbf{Neural Network Approximations:} Several studies \cite{AbuKhalaf2005,Kamalapurkar2018,Vamvoudakis2009,Vamvoudakis2010,Vrabie2009a} have employed neural networks to approximate the VFs.
    \item \textbf{Galerkin Approximation Methods:} Galerkin approaches \cite{Kalise2018,Kalise2020,Beard1997} minimize an \(L^2\)-residual of the GHJB equation, providing a framework to approximate the VFs with polynomials.
\end{itemize}
Additionally, hybrid algorithms have been developed that integrate value iteration -- an alternative iterative method for approximating the optimal VF -- with PI to enhance convergence speed \cite{Alla2015}. Furthermore, data-driven variants of PI have been proposed \cite{Luo2014, Kamalapurkar2018}, which eliminate the requirement for an explicit model of the system dynamics by utilizing measured data. These approaches can also be implemented in an online adaptive manner \cite{Vrabie2009, Jiang2012, Jiang2014}.
As a computational contribution, we present numerical experiments based on a grid-based PI algorithm, inspired by the method proposed in \cite{Alla2015}. The main distinction of our approach lies in the use of a kernel-based interpolation technique in place of orthogonal polynomials. These experiments quantitatively demonstrate the practical relevance of the theoretical issues discussed, particularly their influence on approximation accuracy. Furthermore, our results suggest that certain analytical assumptions may be overly restrictive in many practical applications.\\
The remainder of the paper is organized as follows. Section 2 discusses the solvability of the GHJB equation and analyzes the regularity properties of its solution. In Section 3, we rigorously construct domains on which the updated control policy is admissible. Section 4 presents convergence results for the PI algorithm in two settings: one with a bounded operating domain, and another over the entire state space  
$\mathbb{R}^N$.  In Section 5, we provide numerical experiments to illustrate the theoretical results. Finally, Section 6 concludes the paper with a discussion and outlook.

\section{Solving the generalized Hamilton–Jacobi–Bellman equation }\label{sec:sec2}
\noindent Before addressing the unique solvability of the GHJB equation -- the primary focus of this section -- we first provide a precise definition of what constitutes an admissible controller in the context of the present work. The concept of controller admissibility is fundamental to the study of the PI algorithm and is inherently linked to the domain over which it must be valid. Accordingly, we present the following definition.

\begin{definition}[Admissible Controller--Domain Pair]
\label{def:admissible}
Let \(\Omega \subset \mathbb{R}^N\) be a domain and let { \(u \in \xCzero(\Omega,\mathbb{R}^M)\)} be a controller.  
We say that the pair \((u,\Omega)\) is admissible for \eqref{eq:MPInfinit} if the following conditions hold:
\begin{enumerate}
  \item \textbf{Containment of the Origin.} 
  The domain \(\Omega\) contains the origin, i.e., \(0 \in \Omega\).
  \item \textbf{Zero Control at the Origin.}
  The controller satisfies
  $    u(0) = 0.
  $
    \item \textbf{Forward Invariance of \(\Omega\).}
  For every \(x \in \Omega\), the solution \(\mathbf{x}_u(t;x)\) remains in \(\Omega\) for all \(t \ge 0\), that is,
  $
    \mathbf{x}_u(t;x) \in \Omega,$
   for all $ t \ge 0.
  $
  \item \textbf{Asymptotic Stabilization.} 
  The controller \(u\) asymptotically stabilizes the system on \(\Omega\), i.e.,
  \[
    \lim_{t\to\infty} \Vert \mathbf{x}_u(t;x)\Vert \;=\; 0,
    \quad \forall\, x \in \Omega.
  \]

  \item \textbf{Local Exponential Stabilization.} 
  The controller \(u\) provides local exponential stability in a neighborhood of the origin. Concretely, there exist constants \(\varepsilon > 0\), \(C > 0\), and \(\nu > 0\) such that for all \(x \in B_\varepsilon(0) \subset \Omega\), the trajectory \(\mathbf{x}_u(t;x)\) satisfies
  \[
    \|\mathbf{x}_u(t;x)\| \;\le\; C\,e^{-\nu\,t}\,\|x\|,
    \quad \forall\, t \ge 0.
  \]
\end{enumerate}
\end{definition}
\noindent Throughout the paper, we use $\Vert  \cdot  \Vert$ to denote the Euclidean norm. The local exponential stabilization property is crucial for ensuring that the associated VF $v_u$ defined in \eqref{eq:VFu} is finite and, thus, well-defined. This will be demonstrated later in the proof of Part (1) of Theorem \ref{thm:GHJB}.
Furthermore, note that sufficient conditions for the existence of a local exponentially
stabilizing controller can be checked by examining the pair \((Df(0), g(0))\) and applying 
linearization arguments, where \(Df(0)\) denotes the Jacobian matrix of the vector field 
\(f\) evaluated at the equilibrium \(x = 0\). Specifically, if there exists a matrix 
\(K_R \in \mathbb{R}^{N \times N}\) such that
\[
  A \;:=\; Df(0)\;-\;\tfrac12\,g(0)\,R^{-1}\,g(0)^{\top} K_R
\]
is Hurwitz, i.e.
\[
  \operatorname{Re}\bigl(\sigma(A)\bigr)\;\subset\;(-\infty,0),
\]
where \(\sigma(A)\) denotes the spectrum of \(A\), then, as a direct consequence of
\cite[Corollary 4.3]{Khalil2002}, the feedback law
\[
  u_{K_R}(x)\;:=\;-\tfrac12\,R^{-1}\,g(x)^{\top} K_R x
\]
locally exponentially stabilizes the system in a (sufficiently small) neighborhood of the steady-state zero.\\
Next, we present two lemmas. The first lemma establishes,  for a given admissible controller–domain pair $\bigl(u,\Omega\bigr),$ the uniqueness of the solution to the GHJB equation by analyzing its homogeneous part.
\begin{lemma}\label{lemma:LemmahomogenSol}
    Let $\bigl(u,\Omega\bigr)$ be an admissible controller--domain pair for \eqref{eq:MPInfinit}. If {$v \in \xCone(\Omega,\mathbb{R})$} satisfies
    \begin{align}
        \langle f(x) + g(x)\,u(x), \nabla v(x) \rangle &= 0 \quad \text{for all } x \in \Omega, \label{eq:invariance}\\[1mm]
        v(0) &= 0, \label{eq:bc}
    \end{align}
    then it must hold that
    \[
        v(x) = 0 \quad \text{for all } x \in \Omega.
    \]
\end{lemma}
\begin{proof}
    See Appendix \ref{sec:LemmahomogenSolProof}.
\end{proof}
\noindent Thus, under the assumption of the preceding lemma, any two solutions of the GHJB equation \eqref{eq:GHJB} must coincide. This uniqueness argument is fundamental to the proof of the subsequent  Theorem \ref{thm:GHJB}.  Furthermore, we will  make use of the smoothness of the flow map \(x_u(t;x)\) with respect to the initial state \(x\), as well as local exponential stabilizability estimates for the corresponding derivatives. These results will be stated and proved in the following lemma.
\begin{lemma}\label{lem:exp_decay_higher_derivatives}
  Let { \(l \in \xCinfty(\Omega,\mathbb{R}^N)\)} for a domain $\Omega \subset \mathbb{R}^N$ and consider the ODE
  \begin{equation}\label{eq:ode}
    \dot{\mathbf{x}}(t;x) = l\bigl(\mathbf{x}(t;x)\bigr),
    \quad \mathbf{x}(0) = x \in \Omega,
  \end{equation}
  with a forward invariance property, i.e.
  $\mathbf{x}(t;x) \in \Omega$ for all $t\geq 0$.
  Then, the following statements hold:
  \begin{enumerate}
    \item The flow map is infinitely differentiable in both time and initial state:
  \[
      \mathbf{x}(\,\cdot\,;\, \cdot \,) \in \xCinfty\!\left([0,\infty) \times \Omega,\mathbb{R}^N \right).
    \]
    \item Assume  that \(D l(0)\in \mathbb{R}^{N\times N}\) is Hurwitz. 
    Fix an initial state {\(x \in \Omega\)}   and assume that there exists a $\delta >0$ along with constants \(\,t_0 \ge 0,\; C_0 > 0,\; \mu_0>0\,\)  such that
    \begin{equation}\label{eq:expAssumption}
       \|\mathbf{x}(t;y)\| \le C_0\, e^{-\mu_0\,(t - t_0)} 
       \quad \text{for all}\quad t \ge t_0 \text{ and } y \in \cl{B_{\delta}(x)} \subset \Omega,
    \end{equation}
    where \(\operatorname{cl}\) denotes the closure of a set.
    Then, there exist constants 
    $t' \ge 0$ and $ \mu' > 0$ such that for every multi-index  \(\alpha \in \mathbb{N}_0^N\) it holds that  
    \begin{equation}\label{eq:exp_decay_derivatives}
       \bigl\|\partial^\alpha_y \mathbf{x}(t;y)\bigr\|
       \le C_\alpha\, e^{-\mu'\,(t - t')} 
       \quad \text{for all}\quad t \ge t'\text{ and } y \in \cl{B_{\delta}(x)},
    \end{equation}
     where the constant \( C_\alpha>0\) depends on $\alpha$, but is independent of $y$. 
    \end{enumerate}
 \end{lemma}
\begin{proof}
The proof is provided in Appendix \ref{sec:AppAuxLemma}
\end{proof}
\noindent In the following theorem we establish the smoothness of the VF under the assumption that the corresponding control signal is smooth. This property is essential for the well-posedness of the iterations in Algorithm \ref{algo:PI}: it ensures that the gradient of every VF arising in the iterative process exists, so the subsequent controllers are well-defined.
 \begin{theorem}\label{thm:GHJB}
Suppose that Assumption~\ref{as:data}  holds. Let $\bigl(u,\Omega\bigr)$ be an admissible controller--domain pair for \eqref{eq:MPInfinit} where
$
u \in \xCinfty(\Omega, \mathbb{R}^M ).
$
Then, the following statements hold:
\begin{enumerate}
    \item The VF $v_u$ defined in \eqref{eq:VFu}   satisfies 
    $
    v_u \in \xCinfty(\Omega, \mathbb{R} ).
    $
    \item The function $v_u$ is positive-definite,  the unique solution of the  GHJB equation associated to the controller $u$ and $$\nabla v_u(x) \neq 0 \text{ for all } x \in \Omega \! \setminus\! \{0\}.$$
   \item The updated controller \[
u^+(x) = -\frac{1}{2} R^{-1} g(x)^{\top} \nabla v_u(x).
\]   satisfies \(u^+\in \xCinfty(\Omega,\mathbb{R}^M)\).
\end{enumerate}
\end{theorem}
\begin{proof}
First, we verify that all assumptions of Lemma~\ref{lem:exp_decay_higher_derivatives} are satisfied for the closed-loop system under the controller \(u\). 
Let \(x \in \Omega\) be arbitrary. We consider the ODE
\[
  \dot{\mathbf{x}}_{u}(t;x)
  \;=\;
  f\bigl(\mathbf{x}_{u}(t;x)\bigr)
  \;+\;
  g\bigl(\mathbf{x}_{u}(t;x)\bigr)\,
  u\bigl(\mathbf{x}_{u}(t;x)\bigr)
  \;=:\;
  l\bigl(\mathbf{x}_{u}(t;x)\bigr),
\]
with \(\mathbf{x}_{u}(0;x) = x\). By assumption, {\(l\in \xCinfty(\Omega,\mathbb{R}^N)\)} and, due to the forward invariance property associated with \(u\), it follows from Lemma~\ref{lem:exp_decay_higher_derivatives} that
\[
  \mathbf{x}_{u}(\,\cdot\,;\,\cdot\,)
  \;\in\;
  \xCinfty\bigl([0,\infty)\times\Omega,\mathbb{R}^N\bigr).
\]
Moreover, since \(u\) locally exponentially stabilizes the system,  the matrix \(Dl(0)\in\mathbb{R}^{N\times N}\) is Hurwitz; see, for example, \cite[Corollary~4.3]{Khalil2002} for details.
Furthermore, because \(u\) is asymptotically stabilizing, there exists a time \(t_0>0\) such that
\[
  \mathbf{x}_{u}(t_0;x)
  \;\in\;
  B_{\varepsilon}(0),
\]
where $\varepsilon>0$ is the constant introduced in Definition \ref{def:admissible} (Part (5)).
By the continuity of \(\mathbf{x}_{u}\bigl(\,\cdot\,;\,\cdot\bigr)\), we can choose \(\delta>0\) such that the following holds for the entire neighborhood:
\[
  \mathbf{x}_{u}(t_0;y) \;\in\; B_{\varepsilon}(0)
  \quad
  \text{for all } y\in \cl{B_{\delta}(x)}.
\]
Therefore, by the local exponential stabilization property of \(u\), there exist constants \(C>0\) and \(u>0\) such that
\[
  \bigl\|\mathbf{x}_{u}(t;y)\bigr\|
  \;\le\;
  \underbrace{
    C
    \;\max_{\tilde{y}\in \operatorname{cl}\left(B_{\delta}(x)\right)}
    \bigl\|\mathbf{x}_{u}(t_0;\tilde{y})\bigr\|
  }_{=:C_0}
  \,e^{-\mu_0\,(t-t_0)},
  \quad
  \forall\,t\ge t_0,
  \quad
  \forall\,y \in \cl{B_{\delta}(x)}.
\]
Thus, all requirements in Part (2) of Lemma~\ref{lem:exp_decay_higher_derivatives} are satisfied.

\medskip

\noindent \textbf{Part (1).}  Define the VF
    \[
    v_u(y) = \int_{0}^{\infty} \Bigl[ \underbrace{   h(\mathbf{x}_{u}(t; y)) + \left\langle u(\mathbf{x}_{u}(t; y)),Ru(\mathbf{x}_{u}(t; y)) \right\rangle}_{=:r(\mathbf{x}_{u}(t; y))}  \Bigr]  \mathrm d t,
    \]
   and observe that $r \in { \xCinfty(\Omega,\mathbb{R})}$. We first show that $v_u(y)$ is finite for all $y \in \cl{B_{\delta}(x)}$. Indeed,
    \begin{align*}
        v_u(y) &= \int_{0}^{\infty}   r(\mathbf{x}_{u}(t; y))   \mathrm d t \\
        & =\int_{0}^{\infty}  \Bigl[ r(\mathbf{x}_{u}(t; y))-r(0) \Bigr]  \mathrm d t\\
        & \leq L_r \int_{0}^{\infty} \|\mathbf{x}_{u}(t; y)\|   \mathrm d t\\
        & \leq L_r \left(\int_{0}^{t_0} \|\mathbf{x}_{u}(t; y)\| \, dt +C_0  \int_{t_0}^{\infty}   e^{-\mu_0(t-t_0)}  \mathrm d t \right)   < \infty,
    \end{align*}
    where $L_{r}$ is the Lipschitz constant of $r$ on the compact set 
\begin{align*}
\Xi_x :=  \left\{ \mathbf{x}(t;y) \, | \, (t,y) \in [0,t_0] \times \cl{B_{\delta}(x) }\right\} \;\cup\; \cl{B_\varepsilon (0)}  \subset \Omega ,
\end{align*}
which is  discussed in more detail in the proof of Lemma \ref{lem:exp_decay_higher_derivatives}.
Next, we prove that $\partial^{\alpha} v_u$ exists in $x$  and is continuous in $x$ for all multi-indices  $\alpha \in \mathbb{N}_0^N$. Since $x \in \Omega $ was arbitrary, we can then conclude that $v_u\in \xCinfty(\Omega,\mathbb{R})$.  In the case $|\alpha|=1$, for $z\in B_{\frac{\delta}{2}}(x)$,  consider
\begin{align*}
&\lim_{\omega \to 0} \frac{1}{\omega}\left(v_u(z+\omega e_{\alpha})-v_u(z)\right) \\ =& \lim_{\omega \to 0} \frac{1}{\omega}\left(\int_{0}^{\infty}   r(\mathbf{x}_{u}(t; z+\omega e_{\alpha}))   \mathrm d t-\int_{0}^{\infty}   r(\mathbf{x}_{u}(t; z))   \mathrm d t\right) \\
=& \lim_{\omega \to 0} \int_{0}^{\infty}  \frac{1}{\omega} \Bigl[  r(\mathbf{x}_{u}(t; z+\omega e_{\alpha}))   -   r(\mathbf{x}_{u}(t; z))  \Bigr] \mathrm d t,
\end{align*}
where the second equality holds for $|\omega | < \frac{\delta}{2}$ as  both  $v_u(z+\omega e_{\alpha})$ and $v_u(z)$ are finite, since $z+\omega e_{\alpha},z \in B_{\delta}(x) $. By Part (2) of Lemma \ref{lem:exp_decay_higher_derivatives} and the mean value theorem, it follows for $|\omega | < \frac{\delta}{2}$  that 
\begin{align}
  \left| \frac{1}{\omega}\left(  r(\mathbf{x}_{u}(t; z+\omega e_{\alpha}))   -   r(\mathbf{x}_{u}(t; z))  \right) \right| 
   \leq & \max_{\theta \in [0,1]} \left\{ \left| D^{(1)}r(\mathbf{x}_{u}(t; z+ \theta\omega e_{\alpha})) \right| \lVert \partial^{\alpha}_y \mathbf{x}_{u}(t; z+ \theta \omega e_{\alpha}) \rVert \right\} \notag\\
   \leq & \max_{\tilde{y} \in \Xi_x} \left | D^{(1)}r(\tilde{y}) \right | p(t) \label{eq:intergabelFunc}
\end{align}
with \begin{align*}
    p(t):= \begin{cases}
        \max_{(\tilde{t},\tilde{y}) \in [0,t'] \times \operatorname{cl}\left(B_{\delta}(x)\right)} \lVert \partial^{\alpha}_y \mathbf{x}_{u}(\tilde{t};\tilde{y} ) \rVert & \text{ for } t < t' \\
        C_{\alpha}  e^{-\mu'(t-t')} &\text{ for } t \geq  t'.
    \end{cases}
\end{align*}
Since $p$ is  an integrable function and the upper bound in \eqref{eq:intergabelFunc} is independent of $z$, the dominated convergence theorem allows interchanging the limit and the integral, yielding  
\begin{align*}
    \partial^{\alpha} v_u(z) = \int_{0}^{\infty}   D^{(1)}r(\mathbf{x}_{u}(t; z))    \partial^{\alpha}_y \mathbf{x}_{u}(t; z)  \mathrm d t
\end{align*}
for all $z\in B_{\frac{\delta}{2}}(x)$ and  in particular for $x$.
Continuity of $ \partial^{\alpha} v_u$ in $x$ can be proved by sequential continuity:   
Let \((z_n)_{n\in\mathbb N}\subset B_{\delta/2}(x)\) be a sequence with
\(\displaystyle\lim_{n\to\infty} z_n = x\).
Because \(z_n\in B_{\delta/2}(x)\) for every \(n\), we have the uniform bound  
\[
  \bigl|D^{(1)} r\bigl(\mathbf x_u(t;z_n)\bigr)\,
        \partial_y^{\alpha}\mathbf x_u(t;z_n)\bigr|
  \;\le\;
  \max_{\tilde y\in\Xi_x}
       \bigl|D^{(1)} r(\tilde y)\bigr|\;p(t),
  \qquad\text{for all }n\in\mathbb N,
\]
where \(p(t)\) is the integrable function introduced earlier.  
Hence the dominated convergence theorem yields
\begin{align*}
  \lim_{n\to\infty}\partial^{\alpha} v_u(z_n)
    &= \lim_{n\to\infty}
       \int_{0}^{\infty}
         D^{(1)} r\bigl(\mathbf x_u(t;z_n)\bigr)\,
         \partial_y^{\alpha}\mathbf x_u(t;z_n)\mathrm d t \\[4pt]
    &= \int_{0}^{\infty}
         \lim_{n\to\infty}
         D^{(1)} r\bigl(\mathbf x_u(t;z_n)\bigr)\,
         \partial_y^{\alpha}\mathbf x_u(t;z_n)\mathrm d t \\[4pt]
    &= \int_{0}^{\infty}
         D^{(1)} r\bigl(\mathbf x_u(t;x)\bigr)\,
         \partial_y^{\alpha}\mathbf x_u(t;x)\mathrm d t
     \;=\;
       \partial^{\alpha} v_u(x),
\end{align*}
where the penultimate equality uses the continuity of the integrand
\(D^{(1)} r\bigl(\mathbf x_u(t;x)\bigr)\,
 \partial_y^{\alpha}\mathbf x_u(t;x)\) in \(x\).
Thus \(\partial^{\alpha} v_u\) is continuous at \(x\).
Because \(x\in\Omega\) was arbitrary, we conclude that the first-order
derivatives of \(v_u\) exist and are continuous on all of \(\Omega\).

\medskip
\noindent The existence and continuity of all higher--order derivatives of \(v_u\) in \(\Omega\)
are established by induction in the multi--index order \(|\alpha|\).
For the induction step, we repeat the exact same argument used in the base case \(|\alpha|=1\):
the higher--order derivatives of \(r\) are represented through the associated
multilinear forms, the mean--value theorem is applied together with Lemma~\ref{lem:exp_decay_higher_derivatives}
to get an integrable function as a uniform upper bound, and the dominated convergence theorem is used to interchange the limit and the integral. Consequently, we can conclude that { \(v_u \in \xCinfty(\Omega,\mathbb{R})\)}.

\medskip

\noindent \textbf{Part (2).}  
Consider an arbitrary initial state \(x\in\Omega\) and a small interval \(\Delta t > 0\). The VF can be decomposed as:
\begin{align}\label{eq:posVF}
v_u(x) = \int_{0}^{\infty} r\bigl(\mathbf{x}_{u}(t;x)\bigr)\mathrm d t 
       = \int_{0}^{\Delta t} r\bigl(\mathbf{x}_{u}(t;x)\bigr)\mathrm d t + v_u\Bigl(\mathbf{x}_{u}(\Delta t;x)\Bigr). 
\end{align}
Dividing the rearranged equality by \(\Delta t\) and taking the limit as \(\Delta t\to 0\) yields
\[
-\left\langle \nabla v_u\bigl(\mathbf{x}_{u}(0;x)\bigr),\, \deriv{t}\mathbf{x}_{u}(0;x) \right\rangle = r\bigl(\mathbf{x}_{u}(0;x)\bigr).
\]
Since \(\mathbf{x}_{u}(0;x)=x\) and 
\[
\deriv{t}\mathbf{x}_{u}(0;x)=f(x)+g(x)u(x),
\]
it follows that
\[
\left \langle f(x)+g(x)u(x) , \nabla v_u(x) \right \rangle + h(x) + \left \langle u(x) ,  R  u(x) \right \rangle = 0.
\]
Thus, \(v_u\) satisfies the GHJB equation associated to $u$. Uniqueness of the solution follows from Lemma~\ref{lemma:LemmahomogenSol}.
Next, note that for the initial condition $x=0$,
\[
\mathbf{x}_{u}(t;0) = 0, \quad \forall\, t\ge 0,
\]
since \(l(0)=0\). Hence, the VF at the origin is
\[
v_u(0)=\int_{0}^{\infty} r(0)\,\mathrm d t = \int_{0}^{\infty} 0\, \mathrm d t = 0.
\]
Also the other direction holds, as
\begin{align*}
v_u(x)=0 &\;\Rightarrow\; \int_{0}^{\infty} r\bigl(\mathbf{x}_{u}(t;x)\bigr)\mathrm d t=0 \\
         &\;\Rightarrow\; r\bigl(\mathbf{x}_{u}(t;x)\bigr)=0 \quad \text{for all } t\ge 0 \text{, since }r\in \xCinfty(\Omega,\mathbb{R}) \text{ and } r(x)\geq 0 \text{ for all } x \in \Omega \\
         &\;\Rightarrow\; \mathbf{x}_{u}(t;x)=0 \quad \text{for all } t\ge 0\text{, because $r$ is a positive-definite function} \\
         &\;\Rightarrow\; x=\mathbf{x}_{u}(0;x)=0.
\end{align*}
Thus, we deduce that \(v_u(x)=0\) if and only if \(x=0\).
Moreover, from \eqref{eq:posVF} and the strict positivity of $r$ for all \(x\in\Omega\setminus\{0\}\), it follows directly that $v_u$ is positive-definite. \\
Furthermore, assume, for the sake of contradiction, that there exists
\[
  x \in \Omega \setminus \{0\}
  \quad\text{such that}\quad
  \nabla v_u(x) = 0.
\]
Since \(v_u\) satisfies the GHJB equation, we have
\begin{align*}
0 = \bigl\langle f(x) + g(x)\,u(x),\,\nabla v_u(x)\bigr\rangle
      + h(x) + u(x)^\top R\,u(x)
  = h(x) + \bigl\langle u(x),\,R\,u(x)\bigr\rangle,
\end{align*}
and hence \(h(x)=0\), since \(R\) is positive-definite.  For $x\neq 0$, 
this contradicts the assumption that \(h\) is positive-definite.

\medskip
\noindent \textbf{Part (3).} The updated controller is given by
\[
  u^+(x) = -\frac{1}{2}R^{-1}g(x)^{\top} \nabla v_u(x),
\]
and since \(g \in \xCinfty(\Omega,\mathbb{R}^{N\times M})\) and \(v_u \in \xCinfty(\Omega,\mathbb{R})\), it immediately follows that 
$
  u^+ \in \xCinfty(\Omega,\mathbb{R}^M).
$
\end{proof}
\noindent If we assume that \((u,\Omega)\) is an admissible controller--domain pair and,  in contrast to the smoothness hypothesis of the previous theorem, that $u$ satisfies the relaxed regularity assumption  \(u \in \xCzero (\Omega,\mathbb{R}^N)\).  Then, for every \(x \in \Omega\), the VF $v_u$ still fulfills
\begin{align*}
v_u(x) = \int_{0}^{\Delta t} \Bigl[ h\left(\mathbf{x}_{u}(t;x)\right)+  \left\langle   u\left(\mathbf{x}_{u}(t;x)\right), R u\left(\mathbf{x}_{u}(t;x)\right)\right \rangle\Bigr] \mathrm d t + v_u\left(\mathbf{x}_{u}(\Delta t;x)\right).
\end{align*}
Consequently, taking the limit as \(\Delta t \to 0\) yields
\begin{align*}
    \lim_{\Delta t \rightarrow 0} \frac{v_u\left(\mathbf{x}_{u}(t;x)\right)-v_u(x) }{\Delta t} &=  \lim_{\Delta t \rightarrow 0} -\frac{\int_{0}^{\Delta t} h\left(\mathbf{x}_{u}(t;x)\right)+  \left\langle   u\left(\mathbf{x}_{u}(t;x)\right), R u\left(\mathbf{x}_{u}(t;x)\right)\right \rangle \mathrm d t }{\Delta t} \\
    &=-h(x)- \left\langle   u(x), R u(x) \right \rangle.
\end{align*}
Thus, the limit exists and represents the Lie derivative \cite{Lee2003} of \(v_u\) along the vector field defined by \(\mathbf{x}_{u}(t;x)\). Therefore, if the term \(\left\langle f(x)+g(x)u(x),  \nabla v_u(x)\right\rangle\) in the GHJB equation is interpreted in this weak (Lie derivative) sense, a unique solution to the GHJB equation can be guaranteed, even under considerably weaker regularity assumptions on \(u\). This observation may explain why the works \cite{Beard1997,Beard1998} consider only continuity
of the controller to be sufficient. However, a significant difficulty arises when defining the updated controller \(u^+\). For \(u^+\) to be well-defined, more than the existence of the Lie derivative of \(v_u\) is required; full differentiability of \(v_u\) is necessary. As demonstrated in the proof of the preceding theorem, the VF \(v_u\) generally inherits the regularity properties of the controller \(u\). Consequently, since the definition of the updated controller \(u^+\) involves the gradient \(\nabla v_u\), it  exhibits a reduction in regularity compared to \(u\). To ensure that the updated controller \(u^+\) retains the same level of smoothness as \(u\), it is necessary to assume that \(u \in \xCinfty(\Omega,\mathbb{R}^M)\), which in turn guarantees that  \(u^+ \in \xCinfty(\Omega,\mathbb{R}^M)\). This is then essential for iteratively using the results within the PI without any degradation in regularity at each step.
However, a straightforward analysis of the PI algorithm based on Theorem \ref{thm:GHJB} is hindered by the fact that the updated  \((u^+,\Omega)\) does not necessarily form an admissible controller–domain pair. In particular, the control 
 \(u^+\) may generate trajectories that do not lie entirely within the domain \(\Omega\). Consequently, it becomes necessary to construct a refined domain \(\Omega^+ \subset \Omega\) that is forward-invariant under the trajectories 
 \(\mathbf{x}_{u^+}(t;x)\). The methodology for constructing such a domain is presented in the following section.
\section{Admissibility of the updated controller}\label{sec:sec3}
\noindent A fundamental component of the subsequent analysis -- and, in particular, of establishing the admissibility of the updated controller -- is the Lyapunov function theory. We therefore begin by recalling the concept of a Lyapunov function:

\begin{definition}[Lyapunov function]
Let \(\Omega \subset \mathbb{R}^N\) be a domain. A function  \(V \in \xCone(\Omega,\mathbb{R})\) is called a Lyapunov function for the system
\begin{align}\label{eq:dynamicSys}
\dot{\mathbf{x}}_l(t;x) = l\bigl(\mathbf{x}_l(t;x)\bigr), \quad \mathbf{x}(0) = x \in \Omega,
\end{align}
if \(V\) is positive-definite, i.e., \(V(x) > 0\) for all \(x \in \Omega \setminus \{0\}\) and \(V(0) = 0\), and if
\begin{align}\label{eq:Lyapo2}
 \deriv{t}  V(\mathbf{x}_{l}(t;x)) \Big|_{t=0} =   \left\langle l(x) ,  \nabla V(x)\right\rangle < 0 \quad \text{for all } x \in \Omega \setminus \{0\}, 
\;\text{ and }\;
 \left\langle l(0) ,  \nabla V(0)\right\rangle  = 0.
\end{align}
\end{definition}
\noindent Lyapunov functions are a key tool in stability analysis. In particular, if there exists a Lyapunov function for the system
\eqref{eq:dynamicSys}
on a domain \(\Omega \subset \mathbb{R}^N\), then one can guarantee the existence of a subset \(\tilde{\Omega} \subset \Omega\) on which the trajectories \(\mathbf{x}_l(t;x)\) are asymptotically stable, i.e.,
\[
\lim_{t\to\infty} \|\mathbf{x}_l(t;x)\| = 0,\quad \forall\, x \in \tilde{\Omega}.
\]
For a proof of this result, see  {\cite[Theorem 4.1]{Khalil2002}}. Furthermore, one can characterize such a subset precisely in terms of proper sublevel sets of the Lyapunov function.
\begin{definition}[Proper Sublevel Set]
Let \(c>0\).  We call a bounded domain \(\Omega_c \subset \mathbb{R}^N\)  a  proper
sublevel set of a positive-definite function  
\(V \in \xCzero\! \left(\cl{\Omega_c},\mathbb{R}\right)\) if
\begin{enumerate}
  \item \(V(x) < c\) for every \(x \in \Omega_c\);
  \item \(V(x) = c\) for every \(x \in \partial\Omega_c\).
\end{enumerate}
\end{definition}
 \noindent In the next lemma, we collect some statements about proper sublevel sets and Lyapunov functions that we will explicitly use later in the construction of the domains on which the PI is performed.
\begin{lemma}\label{lem:subLevel}
The following statements hold.
\begin{enumerate}
  \item[\textbf{1.}] \textbf{Proper sublevel subsets.}
        Let \(\Omega \subset \mathbb{R}^{N}\) be a bounded domain and
        \(V \in \xCone(\Omega,\mathbb{R})\) a positive-definite function with
        \(\nabla V(x)\neq 0\) for every \(x\in\Omega\setminus\{0\}\).
        For each \(c>0\) define
        \[
          \Omega_c:=\{x\in\Omega\mid V(x)<c\}.
        \]
        If \(\cl{\Omega_c}\subset\Omega\), then
        \(\Omega_c\) is a {proper sublevel set} of \(V,\)  $$  \cl{\Omega_c}=\{x\in\Omega\mid V(x)\leq c\} \text{ and } \partial \Omega_c =\{x\in\Omega\mid V(x)= c\}.$$

  \item[\textbf{2.}] \textbf{Shrinking a proper sublevel set.}
        Let \(\Omega_c\) be a proper sublevel set of a positive-definite
        function \(V\in \xCzero\!\left(\cl{\Omega_c},\mathbb{R}\right)\) and let
        \(\nu\in(0,1)\).
        Then
        \[
          \Omega_{\nu c}:=\left\{x\in\cl{\Omega_c}\mid V(x)<\nu c\right\}
        \]
        is a proper sublevel set of \(V\) and
        \(\cl{\Omega_{\nu c}}\subset\Omega_c\).

  \item[\textbf{3.}] \textbf{Forward invariance and asymptotic stability.}
        Let \(\Omega_c\subset\mathbb{R}^{N}\) be a proper sublevel set of a 
        Lyapunov function
        \(V\in \xCone\!\left(\cl{\Omega_c},\mathbb{R}\right)\) for the dynamical
        system
         \eqref{eq:dynamicSys} and let $l \in \xCone\!\left(\cl{\Omega_c},\mathbb{R}^N\right) $.
        Then \(\Omega_c\) is forward-invariant under the flow of
        \eqref{eq:dynamicSys} and the solution \(\mathbf{x}_l(t;x)\) exists for all \(t \geq 0\). Moreover, for every initial state \(x\in\Omega_c\),
        \[
          \lim_{t\to\infty}\bigl\lVert\mathbf{x}_{l}(t;x)\bigr\rVert=0,
        \]
        and $0 \in \Omega_c$.
\end{enumerate}
\end{lemma}
\begin{proof}
The proof is provided in Appendix \ref{sec:AppAuxLemmaProperSublevels}
\end{proof}

\noindent  In our context, where the VFs are employed as Lyapunov functions, we assume initially that the domain \(\Omega\) is bounded to guarantee the compactness of $\cl{\Omega_c}$. 
We are now prepared to state and prove the following theorem, which establishes sufficient conditions for defining a set \(\Omega^+ \subset \Omega\) under which \((u^+,\Omega^+)\) forms an admissible controller--domain pair.
\begin{theorem}\label{thm:GHJB2}
Suppose that Assumption~\ref{as:data} holds and that $\Omega\subset \mathbb{R}^N$ is a bounded domain. Let $\bigl(u,\Omega\bigr)$ denote an admissible controller--domain pair for \eqref{eq:MPInfinit}, where
$
u \in \xCinfty(\Omega,\mathbb{R}^M).
$
Then, the following statements hold:
\begin{enumerate}
\item There exists a proper sublevel set $\Omega_c$ with respect to $v_u$.
    \item For every proper sublevel set $\Omega_c \subset \Omega$ of $v_u$,  the pair $\bigl(u_,\,\Omega_c\bigr)$ is an admissible controller--domain pair for \eqref{eq:MPInfinit}.
    \item For every proper sublevel set $\Omega_c \subset \Omega$ of $v_u$, the pair $\bigl(u^+,\,\Omega_c\bigr)$ is an admissible controller--domain pair for \eqref{eq:MPInfinit}.
\end{enumerate}
\end{theorem}
\begin{proof}
\noindent \textbf{Part (1).}
By Theorem \ref{thm:GHJB}, Part (2), we have $\nabla v_u(x) \neq 0$ for every $x \in \Omega \setminus \{0\}$. Hence, by Lemma \ref{lem:subLevel} Part (1), it is enough to show that there exists a constant  $c>0$ such that the sublevel set
    \[
    \Omega_c \;:=\; \{\, x\in\Omega \mid v_u(x) < c \}
    \]
    satisfies $\cl{\Omega_c}\subset \Omega$.
Since $\Omega$ is an open set containing the origin, there exists a constant $\varphi > 0$ such that the closed ball
$
\cl{B_\varphi(0)}  \subset \Omega.
$
Next, select a constant $c\in \mathbb{R}_+$ satisfying
\[
0 < c < \min_{\|x\| = \varphi} v_u(x).
\]
We now argue by contradiction. Suppose that $\Omega_c$ is not contained in $\cl{B_\varphi(0)} $. Then there exists some $x' \in \Omega \setminus \cl{B_\varphi(0)} $ with $v_u(x') < c$. Since the closed ball $\cl{B_\varphi(0)} $ is a neighborhood of the origin and the system dynamics governed by $f + g\,u$ are asymptotically stable, there exists a time $t'>0$ for which the corresponding trajectory satisfies
\[
\|\mathbf{x}_{u}(t';x')\| = \varphi.
\]
Thus, using the definition of the VF, we have
\[
c>v_u(x') = \int_{0}^{t'} r \bigl(\mathbf{x}_{u}(t;x')\bigr)\mathrm d t + v_u\bigl(\mathbf{x}_{u}(t';x')\bigr)\geq  v_u\bigl(\mathbf{x}_{u}(t';x')\bigr) \geq \min_{\|x\| = \varphi} v_u(x) > c,
\]
which is a contradiction. Therefore, we conclude that
\[
\cl{\Omega_c} \subset \cl{B_\varphi(0)} \subset \Omega.
\]

\noindent \textbf{Part~(2).} 
By assumption, $(u,\Omega)$ is an admissible controller--domain pair.  
Because $\Omega_c \subset \Omega$, to obtain an admissible pair $(u,\Omega_c)$
it suffices -- by Lemma~\ref{lem:subLevel}, Part~(3) -- to show that $v_u$ is a
Lyapunov function on $\Omega_c$.  
This is indeed the case: $v_u$ is already a Lyapunov function on the entire
set $\Omega$.  
Theorem~\ref{thm:GHJB}, Part~(2), guarantees that $v_u$ is positive-definite, and
since $v_u$ satisfies the GHJB equation, we have
\[
  \left.\deriv{t}\,v_u\bigl(\mathbf{x}_u(t;x)\bigr)\right|_{t=0}
  \;=\;
  \bigl\langle f(x)+g(x)u(x),\,\nabla v_u(x) \bigr\rangle
  \;=\;
  -\,h(x)\;-\;\langle u(x),\,R\,u(x)\rangle
  \;<\;0,
\]
by Assumption~\ref{as:data}.  Hence, $v_u$ is a Lyapunov function on
$\Omega_c$, and $(u,\Omega_c)$ is therefore an admissible controller--domain pair. Note that the property of local exponential stabilization of $(u,\Omega_c)$ is immediately transferred to $(u,\Omega_c)$, since $\Omega_c$ is an open subset of $\Omega$ that contains the origin.

\medskip

\noindent \textbf{Part (3).}      
We consider the derivative of \(v_u\) along the trajectory \(\mathbf{x}_{u^+}(\cdot;x)\) of the system \(f+g\,u^+\) for \(x \in \Omega_c\). In particular, we have
\[
\deriv{t} v_u\bigl(\mathbf{x}_{u^+}(t;x)\bigr) = \Bigl\langle \nabla v_u\bigl(\mathbf{x}_{u^+}(t;x)\bigr),\, f\bigl(\mathbf{x}_{u^+}(t;x)\bigr) + g\bigl(\mathbf{x}_{u^+}(t;x)\bigr) u^+\bigl(\mathbf{x}_{u^+}(t;x)\bigr) \Bigr\rangle.
\]
Evaluating at \(t=0\) (and noting that \(\mathbf{x}_{u^+}(0;x)=x\)) gives
\[
\deriv{t} v_u\bigl(\mathbf{x}_{u^+}(0;x)\bigr) = \langle \nabla v_u(x),\, f(x) + g(x) u^+(x) \rangle.
\]
Since the pair \((v_u,u)\) satisfies the GHJB equation, we also have
\[
\langle \nabla v_u(x),\, f(x) + g(x) u(x) \rangle = - h(x) - u(x)^\top R\,u(x).
\]
Therefore,
\begin{equation}\label{eq:theoremLypvs}
\deriv{t} v_u\bigl(\mathbf{x}_{u^+}(0;x)\bigr)
=\langle \nabla v_u(x),\, g(x) \bigl(u^+(x)-u(x)\bigr) \rangle - h(x) - u(x)^\top R\, u(x).
\end{equation}
Using the relation 
\[
u^+(x) = -\frac{1}{2} R^{-1} g(x)^\top \nabla v_u(x),
\]
we can rearrange \eqref{eq:theoremLypvs} as follows:
\begin{align}\label{eq:Lyponov2}
\deriv{t} v_u\bigl(\mathbf{x}_{u^+}(0;x)\bigr)\notag 
=& \Bigl\langle -2R\, u^+(x),\, u^+(x) - u(x) \Bigr\rangle - h(x) - u(x)^\top R\, u(x)\notag\\[1mm]
=& -\| R^{1/2}\bigl(u^+(x)-u(x)\bigr) \|^2 - \| R^{1/2}u^+(x) \|^2 - h(x) \;<\; 0.
\end{align}
Since \(v_u\) is positive-definite, it follows that it is a Lyapunov function for the closed-loop
system \(f+g\,u^{+}\).  Hence, by Lemma~\ref{lem:subLevel}, Part~(3), the
sublevel set \(\Omega_c\) is forward-invariant for \(f+g\,u^{+}\); every
trajectory that starts in \(\Omega_c\) converges asymptotically to the origin,
and the origin itself lies in \(\Omega_c\).
Because \(v_u(0)=0\) and \(v_u\) is positive-definite, the first-order
necessary condition for a minimum gives \(\nabla v_u(0)=0\), so
\(u^{+}(0)=0\).
For local exponential stability, we first note that the Hessian matrix \(H_{v_u}(0) \in \mathbb{R}^{N \times N}\) is symmetric and positive-definite, since the positive-definite function \(v_u\) attains its global minimum at the origin. To establish the local exponential stability of the closed-loop system \(f + g\,u^+\), it is necessary, by \cite[Corollary 4.3]{Khalil2002}, to show that the Jacobian matrix
\[
B := D(f+g\,u^+)(0) \in \mathbb{R}^{N\times N}
\]
is Hurwitz. To this end, consider the following Taylor expansions about \(x=0\):
\begin{align*}
\nabla v_u(x) &= H_{v_u}(0)x + O(\|x\|^2), \\
 f(x) + g(x) u^+(x) &= Bx + O(\|x\|^2),\\
 u^+(x) - u(x) &= D(u^+ - u)(0)\,x + O(\|x\|^2),\\
 u^+(x) &= D{u^+}(0)\,x + O(\|x\|^2),\\
 h(x) &= \frac{1}{2} \left\langle   x, H_h(0)x \right \rangle + O(\|x\|^3)
\end{align*}
as $x \to 0$,
where the expansion for \(h\) is valid since \(h\) is a positive-definite function. This further implies that \(H_h(0)\) is positive-definite. Rearranging \eqref{eq:Lyponov2} yields the equality
\[
\langle \nabla v_u(x),\, f(x) + g(x) u^+(x) \rangle + \| R^{1/2}(u^+(x)-u(x)) \|^2 + \| R^{1/2}u^+(x) \|^2 + h(x) = 0,
\]
and substituting the Taylor expansions, we obtain
\begin{align*}
  0 = &\langle H_{v_u}(0)x,\, Bx \rangle\\ &+ \frac{1}{2} \left\langle   x, \left( H_h(0) + 2\left( D(u^+ - u)(0)\right)^{\top} R\,D(u^+ - u)(0) + 2\left( D{u^+}(0)\right)^{\top} R\, D{u^+}(0) \right)x \right \rangle  \\&+ O(\|x\|^3)  
\end{align*}
as $x \to 0$.
Defining
\[
Q := H_h(0) + 2\left(D(u^+ - u)(0)\right)^{\top} R\,D(u^+ - u)(0) + 2 \left(D{u^+}(0)\right)^{\top} R\, D{u^+}(0),
\]
and noting that \(Q\) is positive-definite (since \(H_h(0)\) is positive-definite and both \(\left(D{u^+}(0)\right)^{\top} R\, D{u^+}(0)\) and \(\left(D(u^+ - u)(0)\right)^{\top} R\,\left(D(u^+ - u)(0)\right)\) are positive semidefinite), we deduce that
\[
H_{v_u}(0)^\top B + B^\top H_{v_u}(0) + Q = 0.
\]
Since \(H_{v_u}(0)\) and \(Q\) are positive-definite, and using  {\cite[Theorem 4.6]{Khalil2002}} implies that the matrix \(B\) is Hurwitz. Consequently,  due to  \cite[Corollary 4.3]{Khalil2002}, the closed-loop system \(f + g\,u^+\) is locally exponentially stable.
\end{proof}

\noindent A key observation is that the pair \((u^+, \Omega_c)\) forms an admissible controller--domain pair as the VF \(v_u\) continues to serve as a Lyapunov function for the  dynamics \(f+g\,u^+\). In this context, any proper sublevel set of \(v_u\) in $\Omega$ is sufficient to yield a forward-invariant set for these dynamics. However, for the closed‑loop system governed by the subsequent controller  \[
u^{++}(x) := -\frac{1}{2} R^{-1} g(x)^{\top} \nabla v_{u^+}(x),
\]  \(v_u\) is no longer guaranteed to be a Lyapunov function.
Hence, within the PI scheme, no single proper sublevel set of a previously computed VF remains invariant across all iterations, forcing the computational domain to shrink at each step.

\section{Convergence of the Policy Iteration}

\noindent We now establish both, a well-posedness version of
Algorithm~\ref{algo:PI} for bounded domains and the convergence of the associated VF sequence
\(\{v_{s}\}_{s\in\mathbb{N}}\).
The proof proceeds inductively, invoking
Theorems~\ref{thm:GHJB} and~\ref{thm:GHJB2} at each step.
A technical difficulty arises because, at iteration \(s\in \mathbb{N}\), the GHJB equation
is solved on the forward-invariant open set \(\Omega_{s}\) (here the subscript $s$ denotes now the iteration index and no longer the  sublevel constant) for the dynamics to 
the current feedback law \(u_{s}\), whereas forward invariance of the dynamics for the 
{updated} law \(u_{s+1}\) can be verified only on closed proper sublevel sets
whose boundaries are strictly contained in~\(\Omega_{s}\) (see Part (3) of Theorem \ref{thm:GHJB2}).
To reconcile these domains we apply Lemma~\ref{lem:subLevel}, Parts (1) and (2),
and construct a nested, decreasing family of forward-invariant domains, which
we detail below.
\begin{definition}\label{def:setContruct}
Let \(\Omega_{-1} \subset \mathbb{R}^N\) be a
bounded domain, and let \(\{v_s\}_{s\in\mathbb{N}}\) be a sequence of
positive–definite functions (these will be the VFs generated by PI in the following). For $s\in \mathbb{N}_0$ and $\nu \in (0,1)$, define recursively
\begin{equation}\label{def:c002}
\Omega_{s}^{(\nu)} := \{ x \in \Omega_{s-1} \,|\, v_s(x) < c_s^{(\nu)} \}
\end{equation}
with
\begin{equation}\label{def:c02}
c_s^{(\nu)} := \nu \cdot \sup\left\{ c \in \left(0,(1-\nu)^{-1}\right] \,\middle|\,\cl{ \left\{x \in \Omega_{s-1} \, | \, v_s(x) < c\right\} } \subset \Omega_{s-1} \right\}
\end{equation}
and
\begin{align}\label{def:c03}
    \Omega_{s}= \bigcup_{\nu\in(0,1)}\Omega_{s}^{(\nu)}.
\end{align}
\end{definition}

  \setcounter{algorithm}{1}
\begin{algorithm}[H]
\caption{Domain–aware policy iteration (bounded setting)}\label{algo:PIBounded}
\begin{flushleft}
\textbf{Input:}\\[2pt]
\hspace*{1.5em}%
initial feedback \(u_{0}\colon\Omega_{-1}\to\mathbb{R}^{M}\); bounded domain \(\Omega_{-1}\subset\mathbb{R}^{N}\); tolerance \(\varepsilon>0\)\\[6pt]
\textbf{Initialisation:}\\[2pt]
\hspace*{1.5em}%
set \(v_{-1}\equiv 0\), \(e_{0}:=\varepsilon+1\), and \(s:= 0\) \\[6pt]
\textbf{Main loop:}  \\[2pt]
\textbf{while} \(e_{s}>\varepsilon\) \textbf{do}\\
\hspace*{1.5em}%
\textbf{(1)}\; solve
\(\mathrm{GHJB}\bigl(v_{s},u_{s}\bigr)=0\) on \(\Omega_{s-1}\) with \(v_{s}(0)=0\) for the value function $v_s$;\\
\hspace*{1.5em}%
\textbf{(2)}\; set
\(u_{s+1}(x):=-\tfrac12 R^{-1}g(x)^{\!\top}\nabla v_{s}(x)\);\\
\hspace*{1.5em}%
\textbf{(3)}\; compute the error
\(e_{s+1}:=\lVert v_{s}-v_{s-1}\rVert\);\\
\hspace*{1.5em}%
\textbf{(4)}\; determine \(\Omega_{s}\) by \eqref{def:c002}--\eqref{def:c03} using $v_s$;\\
\hspace*{1.5em}%
\textbf{(5)}\; increment the counter \(s:= s+1\).\\
\textbf{end while}\\[6pt]
\textbf{Output:}\\[2pt]
\hspace*{1.5em}%
approximate optimal value function \(v_{\,s-1}\), feedback \(u_{s}\) and domain $\Omega_{s-1}$.
\end{flushleft}
\end{algorithm}

\noindent With the nested domain construction of Definition~\ref{def:setContruct} at
our disposal, we can refine Algorithm~\ref{algo:PI} so that, at every iteration, GHJB equation is solved on an explicitly
specified region that is provably forward–invariant under the dynamics using current
feedback controller.  The resulting domain-aware PI scheme is
presented in Algorithm~\ref{algo:PIBounded}.
Invoking Theorems~\ref{thm:GHJB} and~\ref{thm:GHJB2}, the next result shows
that Algorithm~\ref{algo:PIBounded} is well-posed under the hypotheses
enumerated therein.

\setcounter{theorem}{2}
\begin{theorem}\label{thm:PIwell-posed}
Let Assumption~\ref{as:data} be satisfied and let \(\Omega_{-1}\subset\mathbb{R}^{N}\) be a bounded domain.
Suppose the initial feedback satisfies
\(u_{0}\in\xCinfty \! \left(\Omega_{-1},\mathbb{R}^{M}\right) \) and that the pair
\(\bigl(u_{0},\Omega_{-1}\bigr)\) constitutes an admissible controller–domain pair
 for~\eqref{eq:MPInfinit}.  
Then the sequence $\{(u_{s},\Omega_{s-1})\}_{s\in\mathbb{N}}$
generated by Algorithm~\ref{algo:PIBounded} satisfies, for every 
$s\in\mathbb{N}_0$,
\begin{enumerate}
  \item $\Omega_{s-1}\subset\mathbb{R}^{N}$ is non-empty;
  \item $u_{s}\in\xCinfty\!\bigl(\Omega_{s-1},\mathbb{R}^{M}\bigr)$;
  \item the pair $(u_{s},\Omega_{s-1})$ is an admissible controller--domain pair for
        problem~\eqref{eq:MPInfinit}.
\end{enumerate}
Consequently, the iterations of Algorithm~\ref{algo:PIBounded} are
well-defined.
\end{theorem}

\begin{proof}
We argue by induction on $s\in\mathbb{N}_{0}$.  
Define the induction hypothesis
\[
\mathcal{P}(s):\quad
\Omega_{s-1}\neq\emptyset,\;
u_{s}\in\xCinfty\!\bigl(\Omega_{s-1},\mathbb{R}^{M}\bigr),\;
\text{and}\;
(u_{s},\Omega_{s-1})\text{ is an admissible controller--domain pair}.
\]
\textbf{Base case $s=0$.} The hypothesis $\mathcal{P}(0)$ holds by the assumptions on the initialization of the algorithm.\\
\textbf{Induction step $s\mapsto s+1$.} Suppose that $\mathcal{P}(s)$ holds for some $s\in\mathbb{N}_{0}$. In Step~4 of Algorithm~\ref{algo:PIBounded}, using $\mathcal{P}(s)$, Theorem~\ref{thm:GHJB} Part (2) yields the unique positive\nobreakdash-definite solution $v_{s}\in\xCinfty\!\bigl(\Omega_{s-1},\mathbb{R}\bigr)$ of the GHJB equation associated with $u_{s}$, and Step~5 provides the updated policy $u_{s+1}\in\xCinfty\!\bigl(\Omega_{s-1},\mathbb{R}^{M}\bigr)$.\\
In Step~8 we determine the new domain $\Omega_{s}$. Because, for $\nu<1$ sufficiently close to $1$, Part~(1) of Theorem~\ref{thm:GHJB2} guarantees that
\[
\bigl\{\,c\in(0,(1-\nu)^{-1}] \bigm| \cl{\{x\in\Omega_{s-1}\mid v_{s}(x)<c\}}\subset\Omega_{s-1}\bigr\}\neq\emptyset,
\]
 making the supremum in~\eqref{def:c02}  well-defined and $\Omega_{s}\neq\emptyset$. Note that $\Omega_{s}$, as a union of open sets, is itself open. \\
Next, we verify that $(u_{s+1},\Omega_{s})$ is an admissible controller--domain pair. Pick an arbitrary $x\in\Omega_{s}$. By construction there exists $\nu_{x}\in(0,1)$ such that $x\in\Omega_{s}^{(\nu_{x})}$. We now show that $\Omega_{s}^{(\nu_{x})}$ is a proper sublevel set of $v_{s}$: Indeed, denoting by $\xi$ the supremum in~\eqref{def:c02}, we can choose a sequence $(\xi_n)_{n \in \mathbb{N}} \subset (0,1)$ with $\lim_{n \rightarrow \infty}\xi_{n}=\xi$ and 
\[
\cl{\{x\in\Omega_{s-1}\mid v_{s}(x)<\xi_{n}\}}\subset\Omega_{s-1}\qquad(n\in\mathbb{N}),
\]
so that each set $\{x\in\Omega_{s-1}\mid v_{s}(x)<\xi_{n}\}$ is a proper sublevel set of $v_s$ by Lemma~\ref{lem:subLevel} Part (1). For $n$ large enough the scaling factor $\tilde{\nu}:=\nu_{x}\xi/\xi_{n}$ satisfies $\tilde{\nu}<1$, and Lemma~\ref{lem:subLevel} Part (2)  (with $\tilde\nu$ playing the role of $\nu$ and $\xi_{n}$ playing the role of $c$ in that lemma)  thus  implies that $\Omega_{s}^{(\nu_{x})}$ is a proper sublevel set of $v_{s}$.\\
Furthermore, since $\Omega_{s}^{(\nu_{x})}$ is a proper sublevel set of $v_s$, Theorem~\ref{thm:GHJB2} Part (3) asserts that $(u_{s+1},\Omega_{s}^{(\nu_{x})})$ is an admissible controller–domain pair. As $x\in\Omega_{s}$ was arbitrary and $\Omega_{s}^{(\nu_{x})}\subset\Omega_{s}$, it follows that $(u_{s+1},\Omega_{s})$ is an admissible controller--domain pair. Hence $\mathcal{P}(s+1)$ holds and thus, by induction, $\mathcal{P}(s)$ is true for every $s\in\mathbb{N}_{0}$.
\end{proof}

\noindent By construction, the sequence of sets $\{\Omega_s\}_{s\in\mathbb{N}}$ is nested: $\Omega_{s+1}\subset \Omega_s$ for every $s\in\mathbb{N}$, since $\Omega_{s+1}^{(\nu)} \subset \Omega_{s}$ for all $\nu \in (0,1)$.  Hence, we define the limiting domain
\[
  \Omega_{\infty}:=\bigcap_{s\in\mathbb{N}_0}\Omega_s,
\]
which represents the largest domain on which all iterates $v_s$ are simultaneously well-defined and smooth. With the limiting domain $\Omega_{\infty}$ established, we are now prepared to state and prove the first convergence theorem for PI on bounded domains.

\begin{theorem}\label{thm:exactPI}
Suppose that Assumption~\ref{as:data} holds and that $\Omega\subset \mathbb{R}^N$ is a bounded domain. Let $\bigl(u_0,\Omega\bigr)$ be an admissible controller--domain pair for \eqref{eq:MPInfinit} with
$
 u_0 \in \xCinfty(\Omega,\mathbb{R}^M). 
$
Then for the sequence of VFs $\{v_s\}_{s\in \mathbb{N}}$ produced by Algorithm \eqref{algo:PIBounded}, the following statements are true:
\begin{enumerate}
    \item For every $s\in\mathbb{N}_0$ and for all $x\in\Omega_s$ it holds
    \begin{equation}
     \label{eq:Beh:mono}   
    v^*(x) \le v_{s+1}(x) \le v_s(x).
    \end{equation}

    \item Suppose there exists a domain $\Omega^* \subset \Omega_{\infty}$, on which the optimal controller $u^*$   exists and under which $u^*$ is forward-invariant, and that for every compact subset $K\subset \Omega^*$, there exist positive constants   
    $C_{\nabla v,K}>0$ and $C_{H_v,K}>0$ such that  the uniform bounds
    \begin{equation}\label{eq:uniformBoundGrad}
    \max_{x\in K}\|\nabla v_s(x)\| \le C_{\nabla v,K} \quad \text{and} \quad \max_{x\in K}\|H_{v_s}(x)\|_2 \le C_{H_v,K}
    \end{equation}
    are satisfied for all $s\in\mathbb{N}_0$, where $H_{v_s}$ denote the Hessian of $v_s$. Then the following  statements hold:
    \begin{itemize}
        \item[(i)]  $v^* \in \xCone(\Omega^*,\mathbb{R}), \quad    \lim_{s\to\infty} \|v_s - v^*\|_{\xCone(K,\mathbb{R})} = 0$;
        \item[(ii)] $ \lim_{s\to\infty} \|u_s - u^*\|_{\xCzero(K,\mathbb{R}^M)} = 0$.
    \end{itemize}

\end{enumerate}
\end{theorem}

\begin{proof}
\noindent \textbf{Part (1).} Fix an arbitrary initial state $x\in\Omega_s$ and consider the trajectory $\mathbf{x}_{u_{s+1}}(t;x)$ corresponding to the controller $u_{s+1}$. Recall from \eqref{eq:Lyponov2} that for all \(t\ge0\),
\begin{align*}
&\deriv{t} v_{s}\bigl(\mathbf{x}_{u_{s+1}}(t;x)\bigr) \\ 
&= -\| R^{1/2}\bigl(u_{s+1}(\mathbf{x}_{u_{s+1}}(t;x))-u_{s}(\mathbf{x}_{u_{s+1}}(t;x))\bigr) \|^2 - \| R^{1/2}u_{s+1}(\mathbf{x}_{u_{s+1}}(t;x)) \|^2 - h(\mathbf{x}_{u_{s+1}}(t;x)).
\end{align*}
By integrating over the interval $[0,T]$, we obtain
\begin{equation}
\label{Beh:1}
\begin{split}
    & v_s\bigl(\mathbf{x}_{u_{s+1}}(T;x)\bigr) - v_s\bigl(\mathbf{x}_{u_{s+1}}(0;x)\bigr) \\
 = & \int_0^T \Bigl[-\| R^{1/2}\bigl(u_{s+1}(\mathbf{x}_{u_{s+1}}(t;x))-u_s(\mathbf{x}_{u_{s+1}}(t;x))\bigr)\|^2 
 +\deriv{t}\Bigl(v_{s+1}\bigl(\mathbf{x}_{u_{s+1}}(t;x)\bigr)\Bigr) \Bigr]\mathrm d t.
 \end{split}
\end{equation}
Further, using the facts that 
$v_s\bigl(0\bigr)=v_{s+1}\bigl(0\bigr)=0$, that $u_{s+1}$ asymptotically stabilizes the system on $\Omega_s$ by Theorem \ref{thm:PIwell-posed}, and taking the limit $T\to\infty$ on both sides of \eqref{Beh:1}, yields
\[
-v_s(x) + v_{s+1}(x) = \int_0^\infty \Bigl[-\|R^{1/2}(u_{s+1}(\mathbf{x}_{u_{s+1}}(t;x))-u_s(\mathbf{x}_{u_{s+1}}(t;x)))\|^2\Bigr]\mathrm d t \le 0.
\]
Thus, we can conclude that
\[
v_{s+1}(x) \le v_s(x) \quad \text{for all } x\in \Omega_s,
\]
as $x \in \Omega_s$ was arbitrary. In addition, the definition of $v_s(x)$ as the cost associated with the control $u_s$ yields
\begin{align*}
v_s(x) &= \int_0^\infty \Bigl[h\bigl(\mathbf{x}_{u_s}(t;x)\bigr) + \|R^{1/2}u_s(\mathbf{x}_{u_s}(t;x))\|^2\Bigr]\mathrm d t\\
 &\ge \inf_{\mathbf{u} \in \mathcal{U}_{\infty}} \int_0^\infty \Bigl[h\bigl(\mathbf{x}_{\mathbf{u}}(t;x)\bigr) + \|R^{1/2}\mathbf{u}(t)\|^2\Bigr]\mathrm d t = v^*(x),
\end{align*}
which completes the proof of Part (1).

\medskip
\noindent \textbf{Part (2).} For any compact subset $K\subset \Omega^* \subset \Omega_\infty$, define the families
\[
\mathcal{F}_K^0 := \{ v_s|_K \in \xCinfty(K,\mathbb{R}) \,:\, s \in \mathbb{N}_0 \}
\]
and, for each multi-index $\alpha\in\mathbb{N}^N$ with $|\alpha|=1$,
\[
\mathcal{F}_K^\alpha := \{ \partial^\alpha v_s|_K \in  \xCinfty(K,\mathbb{R}^N)  \,:\, s \in \mathbb{N}_0 \}.
\]
By Part (1), the sequence $\{v_s(x)\}_{s \in \mathbb{N}_0}$ is monotonically decreasing (pointwise) and thus bounded from above by $v_0(x)$ on $K$. Moreover, the uniform bounds on $\nabla v_s(x)$ and on the Hessian $H_{v_s}(x)$, together with the mean value theorem, ensure that the families $\mathcal{F}_K^0$ and $\mathcal{F}_K^\alpha$ are equicontinuous. By the Arzelà–Ascoli theorem, every sequence in $\mathcal{F}_K^0$ (and in $\mathcal{F}_K^\alpha$) admits a uniformly convergent subsequence. In particular, there exists a subsequence $\{v_{s_\ell}\}_{\ell\in\mathbb{N}}$ and limit functions $\tilde{v}_\infty\in \xCzero(K)$ and $\tilde{v}^{(\alpha)}_\infty\in \xCzero(K)$ (for each $|\alpha|=1$) such that
\[
\lim_{l\to\infty}\|v_{s_\ell} - \tilde{v}_\infty\|_{\xCzero(K)} = 0 \quad \text{and} \quad \lim_{l\to\infty}\|\partial^\alpha v_{s_\ell} - \tilde{v}^{(\alpha)}_\infty\|_{\xCzero(K)} = 0.
\]
Thus, it follows $\tilde{v}_\infty \in   C^{1}(K,\mathbb{R}) $ with  $\partial^\alpha \tilde{v}_{\infty}=  \tilde{v}^{(\alpha)}_\infty$.  As this holds for all compact sets $K \subset \Omega^*$, we can infer  that $\tilde{v}_\infty \in   C^{1}(\Omega^*,\mathbb{R}) $. 
Furthermore, since the pointwise limit 
\[
v_\infty(x):=\lim_{s\to\infty}v_s(x)
\]
exists (by monotonicity and boundedness) for all $x\in K$, uniqueness of limits implies that $\tilde{v}_\infty(x)=v_\infty(x)$ for every $x\in K$.  Thus $v_{\infty} \in \xCzero(K,\mathbb{R})$ and by using  Dini's theorem (which applies due to the monotonicity and continuity of the limit), we can conclude that  the convergence of the entire sequence $\{v_s(x)\}_{s \in \mathbb{N}_0}$  is uniform on $K$. That is 
\[
\lim_{s\to\infty}\|v_s - v_\infty\|_{\xCzero(K)} = 0.
\]
Therefore, every convergent subsequence in $\mathcal{F}_K^\alpha$ must have the same limit function $\partial^\alpha v_{\infty}$ by the uniqueness of the limit and the fact that   $v_\infty \in \xCone(K,\mathbb{R})$. Thus, it follows
\begin{align}\label{eq:inProofvinfty}
   \lim_{s\to\infty}\|v_s - v_\infty\|_{\xCone(K,\mathbb{R})} = 0 .
\end{align}
\noindent Next, we establish that  \(v_\infty\) coincides 
  with the optimal VF $v^*$.  Using \eqref{eq:Beh:mono} from the first part, and by definition, we have
\begin{equation}
\label{eq:Beh:ineq}
    v_\infty(x)=\lim_{s\to\infty}v_s(x) \geq v^*(x).
\end{equation}
Hence, it remains only to show the reverse inequality. From \eqref{eq:inProofvinfty} and the fact that
\[
\langle \nabla v_s(x),\, f(x) + g(x)\,u_s(x)\rangle + h(x) + u_s(x)^\top R\, u_s(x)=0,
\]
with
\[
u_s(x) = -\frac{1}{2} R^{-1}g(x)^\top \nabla v_{s-1}(x)
\]
for all \(x \in \Omega^* \subset \Omega_\infty\), it follows that
\[
\langle \nabla v_\infty(x),\, f(x) + g(x)\,u_\infty(x)\rangle + h(x) + u_\infty(x)^\top R\, u_\infty(x)=0,
\]
where
\[
u_\infty(x) = -\frac{1}{2} R^{-1}g(x)^\top \nabla v_\infty(x) \quad \text{ for all } x \in \Omega^*. 
\]
 For a fixed $x \in \Omega^*$, we now consider the following strictly convex quadratic mapping from $\mathbb{R}^m$ to $\mathbb{R}$:
\[
u \mapsto \langle \nabla v_\infty(x),\, f(x) + g(x)\,u\rangle + h(x) + u^\top R\, u
\]
It can be readily seen, using the first-order optimality condition, that this mapping admits its unique minimizer at \(u_\infty(x)\). Consequently, for any  control \(u\), we get the inequality  
\begin{align*}
 0=&   \langle \nabla v_\infty(x),\, f(x) + g(x)\,u_\infty(x)\rangle + h(x) + u_\infty(x)^\top R\, u_\infty(x) \\
\le&
\langle \nabla v_\infty(x),\, f(x) + g(x)\,u\rangle + h(x) + u^\top R\, u.
\end{align*}
 In particular, for \(u=u^*(x)\) we have
\[
0\le \langle \nabla v_\infty(x),\, f(x) + g(x)\,u^*(x)\rangle + h(x) + u^*(x)^\top R\, u^*(x).
\]
From this, together with the fact that the time derivative of \(v_\infty\) along an optimal trajectory \(\mathbf{x}^*(t;x)\) can be expressed as 
\begin{align*}
\deriv{t}v_\infty\bigl(\mathbf{x}^*(t;x)\bigr)
=&\langle \nabla v_\infty(\mathbf{x}^*(t;x)),\, \dot{\mathbf{x}}^*(t;x)\rangle \\
=&\langle \nabla v_\infty(\mathbf{x}^*(t;x)),\, f(\mathbf{x}^*(t;x)) + g(\mathbf{x}^*(t;x))\,u^*(\mathbf{x}^*(t;x))\rangle,
\end{align*}
we can infer that
\[
0\le \deriv{t}v_\infty\bigl(\mathbf{x}^*(t;x)\bigr) + h(\mathbf{x}^*(t;x)) + u^*(\mathbf{x}^*(t;x))^\top R\, u^*(\mathbf{x}^*(t;x)).
\]
Integrating over the interval \([0,T]\) yields
\[
0\le v_\infty\bigl(\mathbf{x}^*(T;x)\bigr)-v_\infty\bigl(\mathbf{x}^*(0;x)\bigr)
+\int_0^T \Bigl[h\bigl(\mathbf{x}^*(t;x)\bigr) + u^*(\mathbf{x}^*(t;x))^\top R\, u^*(\mathbf{x}^*(t;x))\Bigr]\mathrm d t.
\]
Since \(\mathbf{x}^*(t;x)\to 0\) as \(t\to\infty\) and \(v_\infty(0)=0\) (because \(v_s(0)=0\) for all \(s\in\mathbb{N}\)), letting \(T\to\infty\) yields
\[
v_\infty(x)\le \int_0^\infty \Bigl[h\bigl(\mathbf{x}^*(t;x)\bigr) + u^*(\mathbf{x}^*(t;x))^\top R\, u^*(\mathbf{x}^*(t;x))\Bigr]\mathrm d t = v^*(x).
\]
Finally, combining this with \eqref{eq:Beh:ineq}, we can  conclude that
\[
v_\infty(x)=v^*(x),\quad \text{for all } x\in\Omega^*.
\]
Furthermore, using  the inequality
\begin{align*}
   \|u_s - u^*\|_{C(K)} \leq \max_{x \in K} \left\Vert \frac{1}{2}R^{-1} g(x)^{\top} \right\Vert_2  \|v_s - v^*\|_{\xCone(K,\mathbb{R})},
\end{align*}
we obtain that  
\begin{align*}
    \lim_{s\to\infty} \|u_s - u^*\|_{C(K)} = 0.
\end{align*}
Thus, the proof is complete.
\end{proof}
\noindent  The difficulties associated with nested domains can be circumvented by assuming that \(\Omega\) is forward-invariant under the dynamics for  the suboptimal controllers \(u_s\) for all $s\in \mathbb{N}_0$ and the optimal controller \(u^*\). Under this assumption, one obtains
\[
\lim_{s\to\infty} \|v_s - v^*\|_{\xCone(K,\mathbb{R})} = 0,
\]
for every compact subset \(K \subset \Omega\). However, from a practical perspective, such a forward-invariant set cannot be predetermined since neither the suboptimal controller \(u_s\) nor the optimal controller \(u^*\) is available in advance. An alternative, commonly adopted in the literature, is to assume that \(\Omega = \mathbb{R}^N\), so that the GHJB equation is solved over the entire space \(\mathbb{R}^N\).
Also in this setting, a rigorous analysis of the PI can be carried out -- building directly on the preceding results -- without adopting the assumption, used for example in \cite{Jiang2017}, that  continuously differentiable VFs $v_s$  exists for every $s \in \mathbb{N}$.
In order to ensure that the closures of the sublevel sets of the VFs are compact in this scenario, we impose the following assumption:
\begin{assumption}\label{as:vfRadiallyUndbounded}
  The optimal VF \(v^*\) is radially unbounded, meaning 
  \[
    \lim_{\|x\|\to\infty} v^*(x) \;=\; \infty.
  \]
\end{assumption}
\noindent Under this assumption, \eqref{eq:MPInfinit} is structured such that the optimal cost increases without bound as the initial state moves farther from the target state at the origin.  Exploiting this property, the next theorem establishes the global convergence of Algorithm \ref{algo:PI}, whose GHJB equations are solved on the full state space $\mathbb{R}^N$ at every iteration.
\begin{theorem}\label{thm:exactPI2}
Suppose  Assumptions~\ref{as:data} and ~\ref{as:vfRadiallyUndbounded}  hold.  Let further $\bigl(u_0,\mathbb{R}^N\bigr)$ be an admissible controller--domain pair for \eqref{eq:MPInfinit} with
$
u_0 \in \xCinfty(\mathbb{R}^N,\mathbb{R}^M).
$
Then for the sequence of VFs $\{v_s\}_{s \in \mathbb{N}}$ produced by Algorithm \ref{algo:PI}, where the GHJB equation is solved on $\mathbb{R}^N$, the following statements are true:
\begin{enumerate}
    \item For every $s\in\mathbb{N}_0$ and for all $x\in\mathbb{R}^N$ it holds
    \[
    v^*(x) \le v_{s+1}(x) \le v_s(x).
    \]
    \item Suppose  for every compact subset $K\subset \mathbb{R}^N$ there exist positive constants   
    $C_{\nabla v,K}>0$ and $C_{H_v,K}>0$ such that  the uniform bounds \eqref{eq:uniformBoundGrad}
    are satisfied  for all $s\in\mathbb{N}_0$. If, in addition, an optimal controller $u^*$ exists on $\mathbb{R}^N$, then following  statements hold:
    \begin{itemize}
        \item[(i)] $ v^* \in \xCone(\mathbb{R}^N,\mathbb{R}^M), \quad    \lim_{s\to\infty} \|v_s - v^*\|_{\xCone(K,\mathbb{R})} = 0$;
        \item [(ii)] $ \lim_{s\to\infty} \|u_s - u^*\|_{\xCzero(K,\mathbb{R}^M)} = 0$.
    \end{itemize}
\end{enumerate}
\end{theorem}
\begin{proof}
For the proof, we first require the following auxiliary result:
\begin{gather}\label{eq:auxStatmentR} \tag{Aux}
\text{If } (u, \mathbb{R}^N) \text{ is an admissible controller--domain pair, then also } (u^+, \mathbb{R}^N) \text{ is}.
\end{gather}
We retain the notation from Sections~\ref{sec:sec2} and~\ref{sec:sec3}, where \(u\) denotes the current controller and \(u^{+}\) the updated controller.
Our goal is to recast the setting in a form that permits the application of Part~(3) of Theorem~\ref{thm:GHJB2}.\\
First, we show that the sublevel sets of \(v_{u}\) are bounded.
Suppose, for the sake of contradiction, that a sublevel set of the form
\[
\Omega_c := \{\,y \in \mathbb{R}^{N} \mid v_{u}(y) < c \,\}
\]
is unbounded for a $c>0$.
Then there exists a sequence \(\{x_{n}\}_{n \in \mathbb{N}} \subset \Omega_c\) such that
\[
\|x_{n}\| \to \infty \quad \text{as } n \to \infty.
\]
By the radial unboundedness of the optimal value function \(v^{*}\) we also have
\[
v^{*}(x_{n}) \to \infty \quad \text{as } n \to \infty.
\]
On the other hand, for each \(x_{n} \in \Omega_c\) it follows from the definition of \(v_{u}\) that
\[
c \ge \int_{0}^{\infty} \!\bigl[h\bigl(\mathbf{x}_{u}(t;x_{n})\bigr)
      + u\bigl(\mathbf{x}_{u}(t;x_{n})\bigr)^{\!\top} R\,u\bigl(\mathbf{x}_{u}(t;x_{n})\bigr)\bigr] \,\mathrm{d}t
   \ge v^{*}(x_{n}).
\]
Hence \(v^{*}(x_{n}) \le c\) for all \(n\), contradicting the previous divergence.
Thus the closure of every sublevel set is compact.\\
Furthermore, since \(v_{u} \in \xCinfty\!\bigl(\mathbb{R}^{N},\mathbb{R}\bigr)\) by Part~(1) of Theorem~\ref{thm:GHJB}, continuity implies that \(\Omega_c\) is open and that
\[
\partial\Omega_c :=\{\,y \in \mathbb{R}^{N} \mid v_{u}(y) = c \,\}.
\]
Consequently, for every \(c>0\) the set \(\Omega_c\) is a proper sublevel set of \(v_{u}\).\\
Next, let \(x \in \mathbb{R}^{N}\) be arbitrary and choose constants \(\mu_{1},\mu_{2} \in (1,\infty)\) with \(\mu_{1}>\mu_{2}\).
Set \(c_{1}=\mu_{1}v_{u}(x)\) and \(c_{2}=\mu_{2}v_{u}(x)\).
By construction, \(x \in \Omega_{c_{1}}\) and \(x \in \Omega_{c_{2}}\).\\
Because \(\left(u,\mathbb{R}^{N}\right)\) is an admissible controller–domain pair, the pair \(\left(u,\Omega_{c_{1}}\right)\)is also admissible, since, by Lemma \eqref{lem:subLevel}, Part (3), \(\Omega_{c_{1}}\) is forward invariant under the dynamics \(f+g\,u\) and is an open subset of \(\mathbb{R}^{N}\) that contains the origin. In particular, the local exponential stabilization property transfers directly from \(\left(u,\mathbb{R}^N\right)\) to \(\left(u,\Omega_{c_{1}}\right)\) because  $\Omega_{c_{1}}$ contains a neighborhood of the origin.  \\
Furthermore, since \(\Omega_{c_{2}} \subset \Omega_{c_{1}}\) is a proper sublevel set of \(v_{u}\), Part~(3) of Theorem~\ref{thm:GHJB2} guarantees that \(\left(u^{+},\Omega_{c_{2}}\right)\) is an admissible controller–domain pair (here \(\Omega_{c_{1}}\) plays the role of \(\Omega\) in Theorem~\ref{thm:GHJB2}).
As \(x \in \mathbb{R}^{N}\) was arbitrary, the auxiliary statement~\eqref{eq:auxStatmentR} is established.

\medskip
\noindent \textbf{Part (1).}  We will now turn to the iterative PI scheme. By repeatedly applying the auxiliary statement  \eqref{eq:auxStatmentR}, it follows that \(\left(u_s,\mathbb{R}^N\right)\) is an admissible controller--domain pair for all $s \in \mathbb{N}$.
Using the same argument as in Part (1) of Theorem \ref{thm:exactPI}, but for an arbitrary $x\in\mathbb{R}^N$, we obtain the assertion stated in Part 1 of the current theorem.

\medskip
\noindent \textbf{Part (2).} Via Part (1) of the current theorem, we can conclude -- as in the proof of Part (2) of Theorem \ref{thm:exactPI} -- that there exists a candidate optimal VF \(v_{\infty} \in \xCone\!\left(\mathbb{R}^N,\mathbb{R}\right)\) satisfying
\[
\lim_{s\to\infty}\|v_s - v_\infty\|_{\xCone(K,\mathbb{R})} = 0,
\]
for every compact set \(K \subset \mathbb{R}^N\). Moreover, we have
\[
v_{\infty}(x) \geq v^*(x) \quad \text{for all } x \in \mathbb{R}^N,
\]
and \(v_{\infty}\) satisfies the HJB equation on the entire space \(\mathbb{R}^N\). By an argument analogous to that in Part (2) of Theorem \ref{thm:exactPI},  we deduce that 
\begin{align*}
   v_{\infty}(x) \leq \int_0^\infty \Bigl[h\bigl(\mathbf{x}_{u^*}(t;x)\bigr) + u^*(\mathbf{x}_{u^*}(t;x))^\top R\, u^*(\mathbf{x}_{u^*}(t;x))\Bigr]\mathrm d t = v^*(x), 
\end{align*}
which completes the proof.
\end{proof}
\noindent 
%In contrast to Theorem \ref{thm:exactPI}, in the setting of the preceding theorem, one does not need to assume the existence of an optimal controller; rather, its existence is an inherent outcome of the theorem. Consequently, the optimal controller is given by \eqref{eq:feedback}.\\
Note that Theorem \ref{thm:exactPI} is particularly relevant for the LQR case, where the domain is taken to be \( \mathbb{R}^N\). In this setting, the VFs take the quadratic form 
\[
v_s(x) = \left\langle   x, K_s x \right \rangle.
\]
{Further}, the GHJB equation reduces to a Lyapunov equation that is uniquely solved by a positive-definite matrix \(K_s \in \mathbb{R}^{N \times N}\) (see, e.g., {\cite[Theorem 4.6]{Khalil2002}}). Thus, the VFs must necessarily be quadratic, given the uniqueness of the solution (see Theorem \ref{thm:GHJB} Part (2)). Moreover, due to the monotonic decrease of the VFs along the PI iterations, we have
\[
\left\langle    x, K_{s} x   \right \rangle  \leq  \left\langle    x ,  K_{s-1} x   \right \rangle  \leq \lambda_{\max}(K_{s-1}) \|x\|^2,
\]
which implies that
\[
\Vert K_s\Vert_2 \leq \lambda_{\max}(K_{0})
\]
for all \(s \in \mathbb{N}_0\). This uniform bound on the spectral norm $\Vert \cdot \Vert_2$ of the  gain matrices ensures that condition \eqref{eq:uniformBoundGrad} is satisfied. Consequently, by applying Theorem \ref{thm:exactPI2}, we deduce the convergence of the Newton--Kleinman algorithm, provided that an initial admissible controller--domain pair \((u_0, \mathbb{R}^N)\) is given with  initial controller of the form
\[
u_0(x) := - \frac{1}{2}R^{-1} B^{\top} K_{-1} x
\]
for a matrix $K_{-1} \in \mathbb{R}^{N \times N}$.\\
In the general nonlinear setting under consideration, it is unrealistic  to expect that the GHJB equation can be solved globally or that an initial stabilizing controller exists on the entirety of \(\mathbb{R}^N\), as required by the aforementioned theorem. From this perspective, the construction provided in Definition \ref{def:setContruct} is more practical for achieving convergence via Theorem \ref{thm:exactPI}, as the domains are determined iteratively based on the available VF \(v_s\) at each iteration. However, it is possible that the limiting domain \(\Omega_\infty\) contracts to the singleton \(\{0\}\). This outcome depends on the relative rates at which the domain contracts and the sequence \(\{v_s\}\) converges to \(v^*\). 
In our analysis, the contraction of the domain and the convergence of the VFs -- specifically, the monotonic decrease of the VFs -- occur simultaneously. This allows for the possibility that the limiting domain \(\Omega_\infty\) is not reduced to the singleton $\{0\}$.

\noindent Nevertheless, our results guarantee that the optimal VF \(v^*\) is continuously differentiable on any forward-invariant set \(\Omega^* \subset \Omega_\infty\), provided that conditions \(\eqref{eq:uniformBoundGrad}\) hold; these conditions reflect a well-behaved PI iteration.
\noindent  At the end of this section, we explain why, in many numerical experiments, convergence is observed without the need of explicitly shrinking the domain under consideration.
Specifically, in numerous model problems, particularly those serving as academic benchmarks, the VFs generated by the PI iteration, as well as the optimal VF, may exhibit a high degree of rigidity, in particular, they are real analytic. By the identity theorem for analytic functions, even if the GHJB equations are only solved on a strict subset \(\Omega\) properly, the resulting solution will necessarily coincide with the underlying VF on the entire domain \(\Omega\). 

\section{Numerical experiments}
\label{sec:numerical}
\noindent  For the numerical experiments in this section, we are following the approach proposed in \cite{Alla2015}, where the  \eqref{eq:HBJ2} is solved by means of a grid-based implementation of PI.  Specifically, the spatial domain $\Omega$ is discretized on a grid
$
G := \{ x_1, \ldots, x_n \} \subset \Omega \setminus \{0\},
$
 making this approach applicable only to low-dimensional model problems.
For a fixed controller $u_s$, the policy-evaluation step relies on a
temporal discretization of
\begin{align*}
  v_s(x)
  = \int_{0}^{\Delta t}
      \Bigl[
        h\bigl(\mathbf{x}_{u_s}(t;x)\bigr)
        +\bigl\langle
           u_s\bigl(\mathbf{x}_{u_s}(t;x)\bigr),
           R\,u_s\bigl(\mathbf{x}_{u_s}(t;x)\bigr)
         \bigr\rangle
      \Bigr]\mathrm d t
      + v_s\bigl(\mathbf{x}_{u_s}(\Delta t;x)\bigr),
\end{align*}
using an explicit Euler step for the dynamics and a left-endpoint
quadrature rule for the integral.  
This leads to the problem of finding a function
$\tilde v_s\in \xCzero(\mathbb{R}^N,\mathbb{R})$ that satisfies
\begin{align*}
  \tilde v_s(x)
  &= \Delta t\Bigl[
       h(x)
       +\bigl\langle u_s(x),R\,u_s(x)\bigr\rangle
     \Bigr]
     + \tilde v_s\bigl(
         x + \Delta t\bigl(f(x)+g(x)u_s(x)\bigr)
       \bigr),
     &&\forall x\in\Omega,
\\
  \tilde v_s(0) &= 0.
\end{align*}
Using the grid $G$, we approximate the solution with an
interpolation ansatz.  
More precisely, we seek an interpolant 
$I_{\tilde v_s}\in \xCzero(\mathbb{R}^N,\mathbb{R})$ that fulfills
\begin{align}
  I_{\tilde v_s}(x_i)
  &= \Delta t\Bigl[
       h(x_i)
       +\bigl\langle u_s(x_i),R\,u_s(x_i)\bigr\rangle
     \Bigr]
     + I_{\tilde v_s}\bigl(
         x_i + \Delta t\bigl(f(x_i)+g(x_i)u_s(x_i)\bigr)
       \bigr),
     &&i=1,\dots,n,
  \label{eq:PITabular1}
\\
  I_{\tilde v_s}(0) &= 0.
  \label{eq:PITabular2}
\end{align}
Instead of the orthogonal polynomials used in~\cite{Alla2015},
we employ kernel methods (see \cite{Wendland2004} for a comprehensive introduction).
Here, the surrogate takes the form
\[
  I_{\tilde v_s}(x)=\sum_{i=1}^{n}\alpha_i\,k(x_i,x),
\]
where $k\colon\mathbb{R}^N\times\mathbb{R}^N\to\mathbb{R}$ is a symmetric,
positive-definite kernel.  
Choosing a kernel with $k(0,x)=0$ for every $x\in\mathbb{R}^N$ guarantees that
the boundary condition~\eqref{eq:PITabular2} is automatically satisfied.
Moreover, to solve \eqref{eq:PITabular1}, only a matrix inversion is required. To be more precise, the coefficient vector
$\boldsymbol\alpha=[\alpha_{1},\dots,\alpha_{n}]^{\mathsf T}$ is
obtained from 
\[
  \boldsymbol\alpha
  =\bigl(K_G-K_G^{+}\bigr)^{-1}
   \,\Delta t\,
   \begin{bmatrix}
     h(x_1)+\langle u_s(x_1),R\,u_s(x_1)\rangle\\
     \vdots\\
     h(x_n)+\langle u_s(x_n),R\,u_s(x_n)\rangle
   \end{bmatrix},
\]
where the kernel matrices are defined component-wise by
\[
  (K_G)_{ij} = k(x_i,x_j),
  \qquad
  (K_G^{+})_{ij} = k\!\bigl(x_i,
      x_j+\Delta t\bigl(f(x_j)+g(x_j)u_s(x_j)\bigr)\bigr).
\]
Here $K_G$ is the usual Gram matrix and $K_G^{+}$ captures the Euler
shift. To obtain the updated feedback policy, we compute
\begin{align}
  u_{s+1}(x)
  = -\tfrac12\,R^{-1}g(x)^{\top}\nabla I_{\tilde v_s}(x),
  \label{eq:PITabular3}
\end{align}
where
\[
  \nabla I_{\tilde v_s}(x)
  = \sum_{i=1}^{n}\alpha_i\,\nabla_{2}k(x_i,x),
\]
and $\nabla_{2}$ denotes differentiation with respect to the second
argument of $k$. The grid-based PI algorithm therefore amounts to iteratively   solving~\eqref{eq:PITabular1} for the kernel interpolant
   $I_{\tilde v_s}$, and   updating the feedback values $u_{s+1}$ via~\eqref{eq:PITabular3}, 
until the maximum change of the value-function surrogates
\[
  \max_{i=1,\dots,n}
  \bigl|I_{\tilde v_s}(x_i)-I_{\tilde v_{s-1}}(x_i)\bigr|
\]
falls below a prescribed tolerance.\\
\noindent The model problem analyzed is the Van der Pol oscillator -- a two-dimensional system with a single control input. This problem is a specific instance of \eqref{eq:MPInfinit}, where we set
\[
f(x,y) = \begin{bmatrix} y \\ -x + y(1-x^2) \end{bmatrix}, \quad
g(x,y)= \begin{bmatrix} 0 \\ 1 \end{bmatrix}, \quad h(x,y) = x^2 + y^2,\quad R = \frac{1}{50}.
\]
The initial controller is defined as
\[
u_0(x,y) = -\frac{1}{R}\, g(x,y)^{\top} K_R \begin{bmatrix} x \\ y \end{bmatrix},
\]
where \(K_R \in \mathbb{R}^{2 \times 2 }\) is the solution to the Riccati equation for the linearized version of the problem. This LQR approach involves replacing \(f\) with its linearized form around the origin:
\[
\tilde{f}(x,y)= \begin{bmatrix} 0 & 1 \\ -1 & 1 \end{bmatrix} \begin{bmatrix} x \\ y \end{bmatrix}.
\]
Figure~\ref{fig:vecField} shows the trajectories with initial states chosen on a grid in two different computational domains: 
\(\Omega_{\textbf{A}} := [-1,1] \times [-0.25,0.25]\) (left) and \(\Omega_{\textbf{B}} := [-0.25,0.25] \times [-1,1]\) (right), using the initial controller \(u_0\).
\begin{figure}[htbp]
    \centering
    \begin{minipage}{0.48\textwidth}
         \centering
         \includegraphics[width=\textwidth]{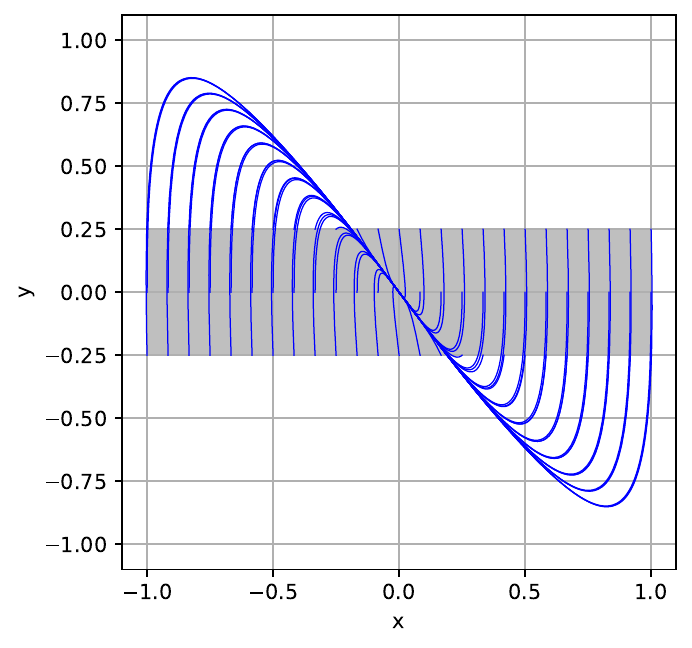}
    \end{minipage}\hfill
    \begin{minipage}{0.48\textwidth}
         \centering
         \includegraphics[width=\textwidth]{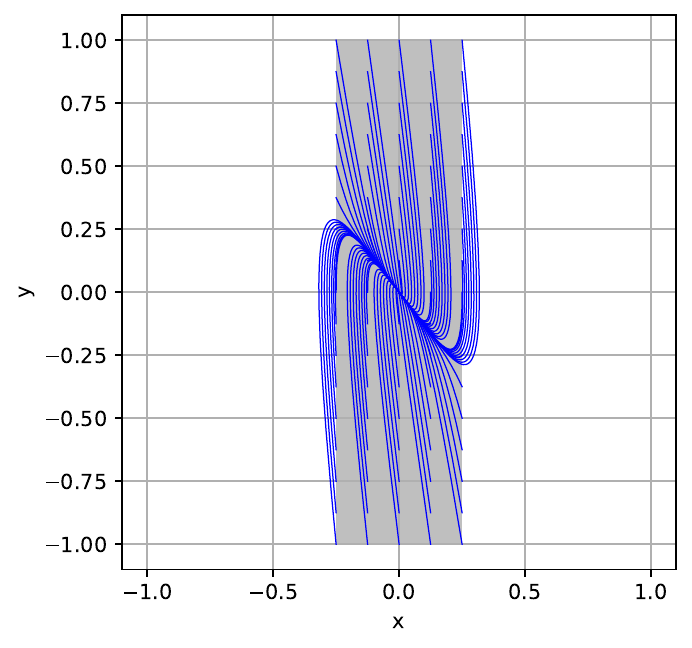}
    \end{minipage}
    \caption{Trajectories of the Van der Pol oscillator model problem using the initial controller \(u_0\) with initial states on the grid on $\Omega_{\textbf{A}}$ (left) and $\Omega_{\textbf{B}}$ (right).}
    \label{fig:vecField}
\end{figure}
It is evident that neither $\Omega_{\textbf{A}}$ nor $\Omega_{\textbf{B}}$ is forward-invariant under the initial controller. However, trajectories originating in $\Omega_{\textbf{B}}$ remain in closer proximity to $\Omega_{\textbf{B}}$ compared to those starting in $\Omega_{\textbf{A}}$. To quantify this difference, we benchmark the grid-based PI algorithm in three distinct scenarios:
\begin{enumerate}
  \item[\textbf{(A)}] fixed computational domain $\Omega_{\textbf{A}}$;
  \item[\textbf{(B)}] fixed computational domain $\Omega_{\textbf{B}}$;
  \item[\textbf{(C)}] an initial computational domain $[-1,1]^2$, which is progressively contracted during the gird-based PI iterations to the set  $\Omega_s$ defined in~\eqref{def:c03} (cf.\ Algorithm~\ref{algo:PIBounded}).
\end{enumerate}
For comparison, the error $\text{E}_{\ell^2}$ is defined as the relative \(\ell^2\)-deviation between the exact optimal VF \(v^*\) and the kernel surrogate solution \(I^n_{v^*}\) obtained after convergence of the PI algorithm:
\[
\mathrm{E}_{\ell^2} := \sqrt{\frac{\sum_{i=1}^{n_{\text{\tiny{test}}}} \left| v^*\bigl(x_{\text{\tiny{test}}}^{(i)}\bigr)-I^n_{v^*}\bigl(x_{\text{\tiny{test}}}^{(i)}\bigr) \right|^2}{\sum_{i=1}^{n_{\text{\tiny{test}}}} \left| v^*\bigl(x_{\text{\tiny{test}}}^{(i)}\bigr) \right|^2}}.
\]
To assess the optimal VF $v^*$ on the test set 
\(\{x_{\text{\tiny{test}}}^{(1)},\ldots,x_{\text{\tiny{test}}}^{(n_{\text{\tiny{test}}})}\}\),
we solve a finite-horizon version of \eqref{eq:MPInfinit} for each initial state $x_{\text{\tiny{test}}}^{(i)}$ $\left(i=1,...,n_{\text{\tiny{test}}}\right)$, via Pontryagin’s maximum principle, using a sufficiently long horizon and a terminal cost $\langle x, K_R x \rangle$; the latter approximates the optimal VF in a neighbourhood of the equilibrium. The resulting open-loop trajectories provide approximate solutions to the infinite-horizon problem \eqref{eq:MPInfinit} and, consequently, approximate values
$v^*\left(x_{\text{\tiny{test}}}^{(i)}\right)$. However, this approximation also introduces an error in the computation of $\mathrm{E}_{\ell^2}$, making it infeasible to reach machine precision.\\
For experiment~\textbf{(A)}, the grid-based PI algorithm is performed on the rectangular grid
\[
  \mathcal{G}^{\mathrm{tr}}_{A} \;=\; 
  G_{1,\,0.25,\,30}\;\subset\;\Omega_{\textbf{A}},
  \qquad 
  G_{a,\,b,\,n} \;:=\;
  \Bigl\{\,\bigl(-a+2a\tfrac{i-1}{n-1},\,-b+2b\tfrac{j-1}{n-1}\bigr)
  \,\Bigm|\, i,j \in \{1,\dots,n\}\Bigr\}.
\]
and the $\mathrm{E}_{\ell^2}$-error is evaluated on the coarser test grid
$\mathcal{G}^{\mathrm{te}}_{A}=G_{1,\,0.25,\,10}\subset\Omega_{\textbf{A}}$.
Experiment~\textbf{(B)} is the axis–flipped analogue:
\[
  \mathcal{G}^{\mathrm{tr}}_{B}=G_{0.25,\,1,\,30}\subset\Omega_{\textbf{B}},
  \qquad
  \mathcal{G}^{\mathrm{te}}_{B}=G_{0.25,\,1,\,10}\subset\Omega_{\textbf{B}}.
\]
Scenario~\textbf{(C)} begins with the full square $[-1,1]^2$:
\[
  \mathcal{G}^{\mathrm{tr}}_{C}=G_{1,\,1,\,30}\subset[-1,1]^2,
  \qquad
  \mathcal{G}^{\mathrm{te}}_{C}=G_{1,\,1,\,10}\subset[-1,1]^2,
\]
and then performs the set construction as in Definition \ref{def:setContruct}.
Figure~\ref{Fig:PI} plots the evolution of the $\mathrm{E}_{\ell^2}$-error,  against the number of PI {iterations} for
all three scenarios.
\begin{figure}[htbp] 
    \centering
         \centering
         \includegraphics[width=\textwidth]{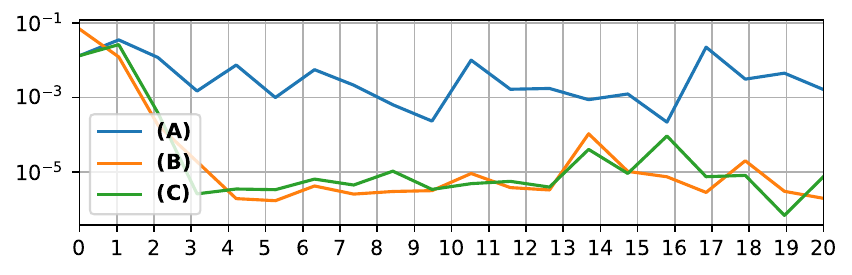}
    \caption{The $\mathrm{E}_{\ell^2}$-error over 20 grid-based PI iterations   for the three scenarios \textbf{(A)}--\textbf{(C)}.}\label{Fig:PI}
\end{figure}
The grid-based PI algorithm in scenario \textbf{(B)} -- the nearly forward-invariant set -- and scenario \textbf{(C)} yields errors two orders of magnitude smaller than in scenario \textbf{(A)}, where the computational domain is far from being forward-invariant. These experiments suggest that forward invariance is crucial for the numerical performance of the PI method. Nevertheless, the errors in scenario \textbf{(B)} are comparable to those in scenario \textbf{(C)}, indicating that explicitly constructing forward-invariant domains as in Algorithem \ref{algo:PIBounded} may not be strictly necessary for convergence.\\
To further illustrate this point, we perform a generalization experiment in which the solution is tested on a grid covering the entire unit square, i.e.,  $G_{1,1,10}$.
For the solution from scenario \textbf{(A)}, the $\text{E}_{\ell^2}$-error obtained after twenty grid-based PI iterations is \(1.83 \times 10^{-3}\)   and for the solution from scenario \textbf{(B)} it is \(1.85 \times 10^{-3}\). This demonstrates that both solutions generalize well on a domain significantly larger than the one used in the PI computations, suggesting that the VFs \(v_s\)  for this model problem possess high regularity -- possibly even analyticity -- on \([-1,1]^2\). If analyticity holds, a strong extension property would follow: although PI is executed properly on a comparatively small forward-invariant subset where convergence is guaranteed, the computed VF after convergence coincides with the optimal VF in a neighbourhood of the origin. By the identity theorem for analytic functions, this local agreement forces the two functions to coincide on the whole square \([-1,1]^2\).\\
A prerequisite for the identity theorem -- namely, that the convergence domain $\Omega_{\infty}$ contains a non-trivial open neighbourhood of the origin -- is examined in a subsequent experiment (Fig. \ref{fig:sublevel}). In that experiment, the sets $\Omega_s$ defined in \eqref{def:c03} and also used in scenario \textbf{(C)} are visualized over ten PI iterations. These sets converge rapidly to a non-trivial, oval-shaped domain, implying that $\Omega_{\infty}$ is strictly larger than the singleton $\{0\}$.

\begin{figure}[htbp]
    \centering
         \centering
         \includegraphics[width=0.5\textwidth]{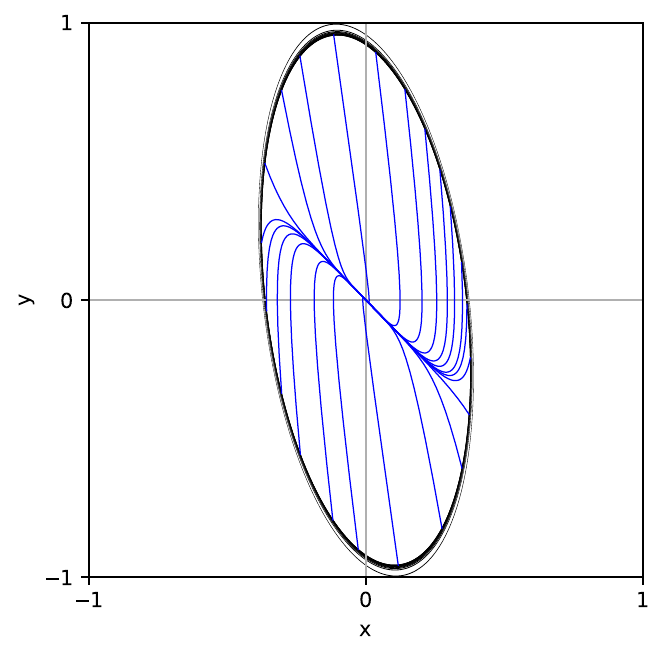}
         % Optional: \caption*{(b)}
    \caption{The nested sets $\Omega_s$ for the Van der Pol oscillator model problem appearing in Algorithm \ref{algo:PIBounded}  for 10 iterations. In the innermost set, obtained from the VF after 10 iterations, the closed‑loop trajectories generated by the corresponding controller -- starting from initial states on the boundary -- are visualized. }
    \label{fig:sublevel}
\end{figure}
\section{Conclusion and outlook}

\noindent The main contribution of the current work is to rigorously address several key issues arising in the PI algorithm for generating an optimal feedback policy in infinite-horizon nonlinear optimal control. These issues include ensuring the differentiability of the solution to the GHJB equation -- which is necessary to guarantee the well-posedness of the PI scheme -- as well as providing a precise definition of the domain over which the GHJB equation can be solved. 
 A potentially undesirable issue arising from our framework is that the iteratively constructed domains may exhibit progressive contraction, possibly degenerating to the trivial singleton $\{ 0 \}$. Consequently, it is important to rigorously analyze the conditions that guarantee that this contraction does not occur and to identify sufficient criteria that prevent such degeneration. These efforts represent a significant direction for future research.\\
However, if the problem is such that this contraction does not occur, convergence is guaranteed on all forward-invariant sets contained within the domains where the GHJB equation is solved,  provided that  \(\eqref{eq:uniformBoundGrad}\) holds. This condition ensures that the PI iterates do not become excessively steep. Moreover, under these assumptions, we can conclude that the optimal VF is continuously differentiable.\\
Although alternative conditions implying
\eqref{eq:uniformBoundGrad} may be conceivable, we suppose that the continuous differentiability of the optimal VF will then be an essential requirement.
In this sense, the statements of Theorem \ref{thm:exactPI} and Theorem \ref{thm:exactPI2} are particularly valuable,  since we can conclude that the optimal VF is continuously differentiable. \\
Furthermore, the numerical findings suggest that the assumptions regarding the domains in which the PI algorithm is performed may be overly restrictive for certain model problems. In the case of the Van der Pol oscillator, for example, the excellent generalization capability of the computed VFs   indicates that they could be analytic on $[-1,1]^2$. It would be of considerable interest in future work to investigate the conditions on the initial controller $u_0$ and the system data \(f\), \(g\), \(h\), and \(R\) that ensure the analyticity of the VFs \(v_s\). Such results would not only explain the observed numerical behavior, but also clarify why explicit domain adjustments appear to be non-critical in practice when implementing PI methods and are therefore often omitted.
Nevertheless, our numerical experiments also provide clear evidence that the choice of computational domain is a significant factor. More precisely, the PI algorithm performs markedly better when the domain is close to being forward-invariant.

\section*{Acknowledgements}
\noindent The authors appreciate and acknowledge Matthias Baur for his helpful comments and insights on this manuscript. Funded by Deutsche Forschungsgemeinschaft (DFG, German Research Foundation) under Project No. 540080351 and Germany’s Excellence Strategy -- EXC 2075 -- 390740016. We acknowledge support by the Stuttgart Center for Simulation Science (SimTech).

%%-----------------------------
   % or any other *.bst* the journal requires
\bibliographystyle{abbrv}   % or: abbrv, amsplain …

\appendix
\section{Proofs}
\subsection{Proof of Lemma \ref{lemma:LemmahomogenSol}}\label{sec:LemmahomogenSolProof}
   \noindent  Since {$v \in \xCone(\Omega,\mathbb{R})$} and $\mathbf{x}_u(t;x)$ is the solution of
    \[
        \dot{\mathbf{x}}_u(t;x) = f(\mathbf{x}_u(t;x)) + g(\mathbf{x}_u(t;x))\,u(\mathbf{x}_u(t;x))
    \]
    with $\mathbf{x}_u(0;x)=x \in \Omega$, the chain rule yields
    \begin{align*} 
      \deriv{t}  v(\mathbf{x}_u(t;x)) &= \langle \nabla v(\mathbf{x}_u(t;x)), \dot{\mathbf{x}}_u(t;x) \rangle \\
        &= \langle \nabla v(\mathbf{x}_u(t;x)), f(\mathbf{x}_u(t;x)) + g(\mathbf{x}_u(t;x))\,u(\mathbf{x}_u(t;x)) \rangle.
    \end{align*}
    By the condition \eqref{eq:invariance}, it follows that
    \[
        \deriv{t} v(\mathbf{x}_u(t;x)) = 0 \quad \text{for all } t \ge 0.
    \]
    Consequently, $v(\mathbf{x}_u(t;x))$ remains constant along trajectories, so 
    \[
        v(\mathbf{x}_u(t;x)) = v(x) \quad \text{for all } t \ge 0.
    \]   
    Since  \((u,\Omega)\) is an admissible controller--domain pair, we have that
  \[
    \lim_{t\to\infty} \Vert \mathbf{x}_u(t;x)\Vert \;=\; 0.
\]
  Thus, { using the continuity of $v$, we can write }
    \[
        v({x}) = \lim_{t \to \infty} v(\mathbf{x}_u(t;x)) = v\left( \lim_{t \to \infty} \mathbf{x}_u(t;x) \right) = v(0).
    \]
    Since  $v(0)=0$ by \eqref{eq:bc}, it follows that
    \[
        v(x) = 0 \quad \text{for every } x \in \Omega,
    \]
 and this concludes the proof. \qed

\subsection{Proof of Lemma \ref{lem:exp_decay_higher_derivatives}}
\label{sec:AppAuxLemma}

\noindent Before proving the lemma, we establish the following auxiliary results, which will be used in the proof.

\begin{lemma}\label{lem:globallyExp}
  Let $A \in \mathbb{R}^{N\times N}$ be a Hurwitz matrix, and suppose 
  $g \in \xCzero \!\left([0,\infty)\times\mathbb{R}^N,\mathbb{R}^N\right)$ is a function that satisfies  
  \begin{equation}\label{eq:g_decay}
      \lVert g(t,y) \rVert \le C_1 e^{-\mu_1 (t-t_1)} \lVert y \rVert
      \qquad \text{for all } t \ge t_1,\text{ and } y\in\mathbb{R}^N,
  \end{equation}
for some constants $C_1>0$, $\mu_1>0$, and an initial time $t_1>0$.  Then there exist constants $C_2>0$, $\mu_2>0$, and $t_2>t_1$ such that the solution of the nonlinear system
  \begin{equation}\label{eq:ode0}
    \dot{\mathbf{x}}(t;x)=A\mathbf{x}(t;x)+g\bigl(t,\mathbf{x}(t;x)\bigr),
    \qquad \mathbf{x}(0)=x\in\mathbb{R}^N,
  \end{equation}
  is globally exponentially stable from time $t_2$ onward; that is,
  \begin{equation}\label{eq:exp_decay}
      \lVert \mathbf{x}(t;x) \rVert \le C_2 e^{-\mu_2 (t-\tilde{t})}
      \lVert \mathbf{x}(\tilde{t};x) \rVert
      \qquad \text{for all } t \ge \tilde{t} \ge t_2.
  \end{equation}
\end{lemma}

\begin{proof}
 From \eqref{eq:g_decay} we know that, for every $\gamma>0$, there exists a
  time $t_\gamma>0$ such that
  \begin{equation}\label{eq:g_small}
      \lVert g(t,y)\rVert \le \gamma \lVert y\rVert,
      \qquad \text{for all } t\ge t_\gamma, \text{ and } y\in\mathbb{R}^N
  \end{equation}
Choosing a sufficiently small $\gamma$ and setting $t_2:=t_\gamma$, the system 
\eqref{eq:ode0}, for every $t\ge t_2$, can be viewed as a Hurwitz linear
  system $\dot{\mathbf x}=A\mathbf x$ perturbed by a bounded nonlinear term whose  norm is at most~$\gamma \lVert y\rVert$. Then, applying arguments similar to those in \cite[Lemma 9.1 and Example 9.1]{Khalil2002}, which address the exponential stability of Hurwitz matrices under small nonlinear perturbations, we conclude that \eqref{eq:exp_decay} holds. This completes the proof.
\end{proof}  
\noindent We are now in a position to prove Lemma \ref{lem:exp_decay_higher_derivatives}.

\begin{proof}[Proof Lemma \ref{lem:exp_decay_higher_derivatives}]
 \noindent\textbf{Part (1).} 
The smoothness statement is a standard result from ODE theory (see, for instance, \cite[Corollary~4.1]{Hartman2002}).

\medskip

\noindent \textbf{Part (2).}  
\noindent The flow map satisfies
\[
\partial_t \mathbf{x}(t;y)= l\bigl(\mathbf{x}(t;y)\bigr) \text{ with } \mathbf{x}(0;y)= y.
\]
For any multi-index $\alpha\in\mathbb{N}_0^N$, we differentiate the above equation with respect to the initial state variable $y$. By interchanging the order of differentiation (which is justified by the smoothness $\mathbf{x}(\, \cdot \, ; \, \cdot \, )$ proved in Part~(1)), we obtain for $\partial^\alpha_y \mathbf{x}(t;y)$ an ODE 
\begin{equation}\label{eq:inLemmaExpo2}
  \partial_t \Bigl(\partial^\alpha_y \mathbf{x}(t;y)\Bigr) = D l\bigl(\mathbf{x}(t;y)\bigr) \, \partial^\alpha_y \mathbf{x}(t;y) + R_\alpha(t;y),
\end{equation}
where the initial condition is given by
\[
\partial^\alpha_y \mathbf{x}(0;y)=
\begin{cases}
e_{\alpha}, & \text{if } |\alpha|=1, \\
0, & \text{if } |\alpha|\geq 2,
\end{cases}
\]
with $e_{\alpha}$ denoting the standard unit vector corresponding to the multi-index $\alpha$ when $|\alpha|=1$. Here, $D l\bigl(\mathbf{x}(t;y)\bigr)$ is the Jacobian matrix of $l$ evaluated at $\mathbf{x}(t;y)$, and the term $R_\alpha(t;y)$ arises from the chain rule. More precisely, $R_\alpha(t;y)$ depends solely on derivatives of $\mathbf{x}(\, \cdot \, ; \, \cdot \, )$ of order strictly less than $|\alpha|$. In particular, one finds that for $|\alpha|=1$
\[
R_\alpha(t;y) \equiv 0,
\]
and for $|\alpha|>1$, the the multivariate Faà di Bruno formula (see \cite{Constantine1996}), yields
\[
R_\alpha(t;y)= \sum_{m=2}^{|\alpha|} \; \sum_{\substack{\beta_1+\cdots+\beta_m=\alpha \\ |\beta_i| < |\alpha|}} \kappa_{\beta_1,\dots,\beta_m}\, D^{(m)} l\bigl(\mathbf{x}(t;y)\bigr) \Bigl( \partial_y^{\beta_1}\mathbf{x}(t;y), \dots, \partial_y^{\beta_m}\mathbf{x}(t;y) \Bigr).
\]
In this expression, $D^{(m)}l\bigl(\mathbf{x}(t;y)\bigr)$ denotes the $m$-th derivative of $l$, which is a symmetric $m$-multilinear form, and the constants $\kappa_{\beta_1,\dots,\beta_m}$ are combinatorial coefficients determined by the differentiation process. 
\medskip

\noindent We now proceed with the proof by induction: \\
For the base case $|\alpha|=1$, consider 
\begin{align}
    \partial_t \Bigl(\partial^\alpha_y \mathbf{x}(t;y)\Bigr)
  &= D l\bigl(\mathbf{x}(t;y)\bigr) \, \partial^\alpha_y \mathbf{x}(t;y) \label{eq:inLemmaExpo1} \\
  &= D l(0) \, \partial^\alpha_y \mathbf{x}(t;y) + \underbrace{\left(D l\bigl(\mathbf{x}(t;y)\bigr) - D l(0)\right) \, \partial^\alpha_y \mathbf{x}(t;y)}_{=:p\bigl(t,\partial^\alpha_y \mathbf{x}(t;y)\bigr)}. \nonumber
\end{align}
This equation can be interpreted as a perturbation of the linear system
\begin{equation}\label{eq:inLemmaHuriwtz}
       \partial_t   \bar{\mathbf{x}}(t;y )   = D l(0)\,  \bar{\mathbf{x}}(t;y ) ,
\end{equation}
where the perturbation term satisfies
\begin{align}\label{eq:inLemma2}
    \|p\bigl(t,\partial^\alpha_y \mathbf{x}(t;y)\bigr)\|  \le \|D l\bigl(\mathbf{x}(t;y)\bigr) - D l(0)\|  \|\partial^\alpha_y \mathbf{x}(t;y)\| \le  L_{Dl} \|\mathbf{x}(t;y)\| \|\partial^\alpha_y \mathbf{x}(t;y)\|.
\end{align}
Here, $L_{Dl}$ denotes the Lipschitz constant of $Dl$ on the compact set
\begin{align*}
 \Xi_x :=   \{ \mathbf{x}(t;y) \, | \, (t,y) \in [0,t_{\varepsilon}] \times \cl{B_{\delta}(x)} \} \cup \cl{B_\varepsilon (0)}  \subset \Omega
\end{align*}
where $\varepsilon>0$ is chosen so that $\cl{B_\varepsilon (0)}\subset\Omega$, and $t_{\varepsilon}>0$  is selected, using the assumption given in \eqref{eq:expAssumption}, so that
\[
C_{0}\,e^{-\mu_{0}(t_{\varepsilon}-t_{0})}\;<\;\varepsilon.
\]
The set $\Xi_x$ is compact because it is the union of two compact subsets of $\Omega$:
\begin{enumerate}
  \item $\cl{B_\varepsilon (0)}$ is compact by construction.
  \item 
    $
      \bigl\{\mathbf{x}(t;y)\mid (t,y)\in[0,t_{\varepsilon}]\times \cl{B_{\delta}(x)}\bigr\}
    $
    is the image of a continuous function under the compact set $[0,t_{\varepsilon}]\times \cl{B_{\delta}(x)}$, and hence is compact.
\end{enumerate}
Moreover, by the assumption given in \eqref{eq:expAssumption} and the estimate
in~\eqref{eq:inLemma2}, we have
\begin{equation}\label{eq:inLemma22}
  \bigl\|p\bigl(t,\partial_y^\alpha \mathbf{x}(t;y)\bigr)\bigr\|
  \le L_{Dl}\,C_0\,e^{-\mu_0\,(t-t_0)}
       \bigl\|\partial_y^\alpha \mathbf{x}(t;y)\bigr\|,
       \qquad
       \text{for all } t\ge t_0,\; y\in\cl{B_\delta(x)}.
\end{equation}
Note that the constants $L_{Dl}$ and $C_0$ are independent of both $y$ and the
multi-index~$\alpha$.  Thus we may apply Lemma~\ref{lem:globallyExp} to the
system~\eqref{eq:inLemmaExpo1}, which has the same dynamics for every
$|\alpha|=1$ (only the initial states differ).  We obtain constants
$\mu'>0$, $t'>0$, and $\tilde{C}>0$ such that
\begin{equation}\label{eq:inLemmaExp3}
    \bigl\|\partial_y^\alpha \mathbf{x}(t;y)\bigr\|
    \le \tilde{C}\,
         \bigl\|\partial_y^\alpha \mathbf{x}(t';y)\bigr\|
         e^{-\mu'(t-t')}
    \le
    \underbrace{\tilde{C}\,
      \max_{\tilde{y}\in\cl{B_\delta(x)}}
      \bigl\|\partial_y^\alpha \mathbf{x}(t';\tilde{y})\bigr\|}_{=:C_\alpha}
      e^{-\mu'(t-t')},
\end{equation}
for all $t\ge t'$ and
$y\in\cl{B_\delta(x)}$.
This establishes the base case for the induction.
\medskip

\noindent \noindent Furthermore, the same argument applies to the state-transition matrix
$\Phi(t,s)$, which satisfies
\begin{equation}\label{eq:inLemmaHomoge}
  \partial_t \Phi(t,s)
  = D l\bigl(\mathbf{x}(t;y)\bigr)\,\Phi(t,s),
  \qquad
  \Phi(s,s)=I,
\end{equation}
for every $t\ge s$.  Observe that the system~\eqref{eq:inLemmaExpo1} and the
column vectors of~\eqref{eq:inLemmaHomoge} share identical dynamics.  Hence,
restricting~\eqref{eq:inLemmaHomoge} to $s\ge t'$, we may once more invoke
Lemma~\ref{lem:globallyExp} and obtain with the same constants as before,
\begin{equation}\label{eq:inLemmaHomogeExpo}
  \|\Phi(t,s)\|_{2}
  \le \|\Phi(t,s)\|_{\mathcal F}
  \le \tilde{C}\,
       \|\Phi(s,s)\|_{\mathcal F}\,
       e^{-\mu'(t-s)}
  \le \tilde{C}\,\sqrt{N}\,e^{-\mu'(t-s)},
  \qquad
  \text{for all } t\ge s\ge t'.
\end{equation}
This result will be instrumental in the subsequent induction step.
\medskip

\noindent Next, in the induction hypothesis, we assume that for some fixed $n\in\mathbb{N}$ the assertion in Part (2) of the lemma holds for every multi-index $\alpha\in\mathbb{N}_0^N$ with $|\alpha|\le n$. For the induction step, let $\bar{\alpha}\in\mathbb{N}_0^N$ be any multi-index with $|\bar{\alpha}| = n+1$. We now consider the system \eqref{eq:inLemmaExpo2} and estimate the remainder term $R_{\bar{\alpha}}(t;y)$. By the definition of $R_{\bar{\alpha}}(t;y)$, we have
\begin{align}\label{eq:inLemmaReminderTerm}
\|R_{\bar{\alpha}}(t;y)\|
  &= \left\|\sum_{m=2}^{|\bar{\alpha}|} \; \sum_{\substack{\beta_1+\cdots+\beta_m=\bar{\alpha} \\ |\beta_i| < |\bar{\alpha}|}} \kappa_{\beta_1,\dots,\beta_m} D^{(m)} l\bigl(\mathbf{x}(t;y)\bigr) \Bigl( \partial_y^{\beta_1}\mathbf{x}(t;y), \dots, \partial_y^{\beta_m}\mathbf{x}(t;y) \Bigr)\right\| \notag \\[1mm]
    &\leq \sum_{m=2}^{|\bar{\alpha}|} \; \sum_{\substack{\beta_1+\cdots+\beta_m=\bar{\alpha} \\ |\beta_i| < |\bar{\alpha}|}} \kappa_{\beta_1,\dots,\beta_m} \Bigl\| D^{(m)}l\bigl(\mathbf{x}(t;y)\bigr)\Bigr\|\prod_{j=1}^{m}\Bigl\|\partial_y^{\beta_j}\mathbf{x}(t;y)\Bigr\| \notag \\[1mm]
    &\le \sum_{m=2}^{|\bar{\alpha}|} \; \sum_{\substack{\beta_1+\cdots+\beta_m=\bar{\alpha} \\ |\beta_i| < |\bar{\alpha}|}} \kappa_{\beta_1,\dots,\beta_m} \Bigl\|D^{(m)}l\bigl(\mathbf{x}(t;y)\bigr)\Bigr\| \prod_{j=1}^{m} \Bigl( C_{\beta_j}\, e^{-\mu' (t-t')} \Bigr) \notag \\[1mm]
    &= e^{-2\mu'(t-t')} \sum_{m=2}^{|\bar{\alpha}|} \; \sum_{\substack{\beta_1+\cdots+\beta_m=\bar{\alpha} \\ |\beta_i| < |\bar{\alpha}|}} \kappa_{\beta_1,\dots,\beta_m}\,\Bigl\|D^{(m)}l\bigl(\mathbf{x}(t;y)\bigr)\Bigr\|\, e^{-(m-2) \mu'(t-t')} \prod_{j=1}^{m} C_{\beta_j} \notag \\[1mm]
    &\le e^{-2\mu'(t-t')} \sum_{m=2}^{|\bar{\alpha}|} \; \sum_{\substack{\beta_1+\cdots+\beta_m=\bar{\alpha} \\ |\beta_i| < |\bar{\alpha}|}} \kappa_{\beta_1,\dots,\beta_m}\,\Bigl\|D^{(m)}l\bigl(\mathbf{x}(t;y)\bigr)\Bigr\| \prod_{j=1}^{m} C_{\beta_j} \notag \\[1mm]
    &\leq C_{R,\bar{\alpha}} \, e^{-2\mu'(t-t')}, \, \text{ for all } t \geq t' \text{ and } y \in \cl{B_{\delta}(x)}.
\end{align}
Here, we used for the first inequality the triangle inequality and the submultiplicative property of the norm for the multilinear form, for the second inequality, the induction hypothesis, valid since $|\beta_j| < |\bar{\alpha}|$ and for the third inequality the fact that $e^{-(m-2)\mu'(t-t')}\le 1$ for $t\geq t'$ and $m\geq 2$. The constant $C_{R,\bar{\alpha}}$ is defined by
\[
C_{R,\bar{\alpha}}  :=  \sum_{m=2}^{|\bar{\alpha}|} \; \sum_{\substack{\beta_1+\cdots+\beta_m=\bar{\alpha} \\ |\beta_i| < |\bar{\alpha}|}} \kappa_{\beta_1,\dots,\beta_m}\, \max_{\tilde{x} \in \Xi}\Bigl\|D^{(m)}l\bigl(\tilde{x}\bigr)\Bigr\| \prod_{j=1}^{m} C_{\beta_j}.
\]
The boundedness of $\max_{\tilde{x} \in \Xi}\Bigl\|D^{(m)}l\bigl(\tilde{x}\bigr)\Bigr\|$ follows from the continuity of $D^{(m)}l$ and the compactness of $ \Xi $.
Next, note that the homogeneous part of \eqref{eq:inLemmaExpo2} is given by \eqref{eq:inLemmaHomoge}
for any multi-index $\alpha$. Hence, by the variation of constants formula, the general solution of \eqref{eq:inLemmaExpo2} can be written for $t\geq t'$ as
\begin{align*}
\partial_y^{\bar{\alpha}}\mathbf{x}(t;y)
 &= \Phi(t,t')\, \partial_y^{\bar{\alpha}}\mathbf{x}(t',y) + \int_{t'}^t \Phi(t,s)\, R_{\bar{\alpha}}(s;y) \mathrm d s.
\end{align*}
Taking the $2$-norm  yields
\begin{align*}
   \Vert \partial_y^{\bar{\alpha}} \mathbf{x}(t;y) \Vert &= \left\Vert \Phi(t,t') \partial_y^{\bar{\alpha}} \mathbf{x}(t',y) + \int_{t'}^t \Phi(t,s) R_{\bar{\alpha}}(s;y) \mathrm d s \right\Vert \\
   &\leq  \Vert \Phi(t,t')\Vert \Vert\partial_y^{\bar{\alpha}} \mathbf{x}(t',y)\Vert + \int_{t'}^t \Vert\Phi(t,s)\Vert  \Vert R_{\bar{\alpha}}(s;y) \Vert \mathrm d s \\
    &\leq   \tilde{C}  \sqrt{N} \left( e^{-\mu' (t-t')}\Vert \Vert\partial_y^{\bar{\alpha}} \mathbf{x}(t',y)\Vert + \int_{t'}^t   e^{-\mu'(t-s)}  C_{R,\bar{\alpha}} e^{-2 \mu' (s-t')} \mathrm d s \right) \\
        &\leq    e^{-\mu' (t-t')} \underbrace{ \tilde{C}  \sqrt{N} \left(   \max_{\tilde{y}\in\cl{B_{\delta}(x)}}\Vert\partial_y^{\bar{\alpha}} \mathbf{x}(t',\tilde{y})\Vert +e^{\mu'
         t'}  C_{R,\bar{\alpha}}  \int_{t'}^t   e^{-\mu' s}   \mathrm d s \right)}_{=:C_{\bar{\alpha}}< \infty}, 
\end{align*}
where we used for the first inequality the triangle inequality and the submultiplicative property of the norm, and, for the second inequality, the estimates \eqref{eq:inLemmaHomogeExpo} and \eqref{eq:inLemmaReminderTerm}. Overall, we deduce that
\[
\|\partial_y^{\bar{\alpha}}\mathbf{x}(t;y)\| \le C_{\bar{\alpha}}\, e^{-\mu'(t-t')},  \, \text{ for all } t \geq t' \text{ and } y \in \cl{B_{\delta}(x)}.
\]
This completes the induction step. By the principle of complete induction, the assertion in Part~(2) of the lemma is thereby established for all multi-indices $\alpha\in\mathbb{N}_0^N$.
\end{proof}

\subsection{Proof of Lemma \ref{lem:subLevel}}
\label{sec:AppAuxLemmaProperSublevels}
\begin{proof}
    \noindent \textbf{Part (1).}  Let  $y\in\cl{\Omega_c}$.
There exists a sequence \((y_i)_{i\in\mathbb{N}}\subset\Omega_c
\subset\cl{\Omega_c}\subset\Omega\) such that
\(\lim_{i\to\infty}y_i=y\).
Since \(V\) is continuous on \(\Omega\),
\[
\lim_{i\to\infty}V(y_i)<c \;\Longrightarrow\; V(y)\le c,
\]
so \(y\in\{x\in\Omega\mid V(x)\le c\}\).
Hence
$
\cl{\Omega_c}\subset\{x\in\Omega\mid V(x)\le c\}.
$\\

\noindent  Conversely, let \(y\in\{x\in\Omega\mid V(x)\le c\}\).
If \(V(y)<c\), then \(y\in\Omega_c\subset\cl{\Omega_c}\).
Hence it remains to prove the case \(V(y)=c\). Because $c>0$ and $V$ is positive-definite, we have $y \neq 0$ and therefore, by assumption, \(\nabla V(y)\neq 0\). We now define
\[
p:=\frac{\nabla V(y)}{\lVert\nabla V(y)\rVert^{2}}\qquad \text{ and } \qquad
q(s):=-\frac{\langle\nabla V(y-sp),\nabla V(y)\rangle}{\lVert\nabla V(y)\rVert^{2}}.
\]
The fundamental theorem of calculus yields
\begin{align}\label{eq:inLemmaSub}
 V(y-tp)=
V(y)-\int_{0}^{t}\langle\nabla V(y-sp),p\rangle\,\mathrm ds
=c+\int_{0}^{t}q(s)\,\mathrm ds   
\end{align}
for all \(t\le t_{1}\), where \(t_{1}>0\) is chosen so small that
\(y-tp\in\Omega\) for every \(t\le t_{1}\) (This is possible because \(\Omega\) is open).
Since \(q\) is continuous and \(q(0)<0\), there exists \(t_{2}\in(0,t_{1})\)
with \(q(t)<0\) for every \(t\le t_{2}\).
Thus, using \eqref{eq:inLemmaSub}, it follows \(V(y-tp)<c\)  and therefore \(y-tp\in\Omega_c\) for \(t\le t_{2}\).
Since the sequence \(\bigl(y-n^{-1}t_{2}p\bigr)_{n\in\mathbb N}\subset\Omega_c\)
converges to \(y\), it follows that \(y\in\cl{\Omega_c}\). \\

\noindent Combining the two inclusions we conclude
\[
\cl{\Omega_c}=\{x\in\Omega\mid V(x)\le c\}.
\]
Furthermore, using this identity and the fact that $\Omega_c$  is open (because $V$ is continuous and $\Omega$ is open), we have
$$\partial \Omega_c= \cl{\Omega_c} \setminus \Omega_c =  \{x\in\Omega\mid V(x)\le c\} \setminus \Omega_c  = \{x\in\Omega\mid V(x)= c\},$$
which also implies that the set $\Omega_c$ is a proper sublevel set.

\medskip 
\noindent \textbf{Part (2).} 
By definition of $\Omega_{\nu c}$, the first property of a proper sublevel set,
namely $V(x) < \nu c$ for all $x \in \Omega_{\nu c}$, is satisfied.
With regard to the second property, i.e., $x \in \partial \Omega_{\nu c}$
must satisfy $V(x)=\nu c$, we assume, for the sake of contradiction, that there exists
$y \in \partial \Omega_{\nu c}$ with $V(y) \neq \nu c$.
Because $V$ is continuous on $ \Omega_{\nu c} \subset \cl{\Omega_c}$, the possibility $V(y)>\nu c$  is ruled out; hence we must have $V(y)<\nu c$. We now distinguish two cases:
\begin{enumerate}
  \item $y \in \Omega_c$.  
        Since $\Omega_c$ is open and $V$ is continuous on $\Omega_c$, there exists a
        neighbourhood $\mathcal N_y \subset \Omega_c$ such that
        $V(\tilde y) <\nu c$ for all $\tilde y \in \mathcal N_y$. 
        Hence $y \notin \partial \Omega_{\nu c}$.
  \item $y \in \partial \Omega_c$.  
        In this case $V(y)=c$, contradicting $V(y) < \nu c$.
\end{enumerate}
Thus no such $y$ exists; therefore $V(x)=\nu c$ for all
$x \in \partial \Omega_{\nu c}$. Additionally, since 
$ \cl{ \Omega_{\nu c}} \setminus \partial \Omega_{\nu c} =  \Omega_{\nu c},$
$\Omega_{\nu c}$ is open. Together with the properties above, this shows that $\Omega_{\nu c}$ is a proper sublevel set of $V$.
\\
Furthermore, since $\partial \Omega_{\nu c}\subset \Omega_c$ and
$\Omega_{\nu c}\subset \Omega_c$, it follows that
$\cl{\Omega_{\nu c}} \subset \Omega_c$.

\medskip  
        \noindent \textbf{Part (3).} Assume, for the sake of contradiction, that the set $\Omega_c$ is not
forward-invariant for the dynamical system~\eqref{eq:dynamicSys}.
Then there exists an initial state $y\in\Omega_c$ and, by continuity of the
trajectory $t\mapsto\mathbf{x}_{l}(t;y)$, a first time $t_{B}>0$ such that
\[
  \mathbf{x}_{l}(t_{B};y)\in\partial\Omega_c.
\]
Hence
\[
  c
  \;=\;
  V\!\bigl(\mathbf{x}_{l}(t_{B};y)\bigr)
  \;=\;
  V(y)
  \;+\;
  \int_{0}^{t_{B}}
        \underbrace{%
          \bigl\langle
            l\!\bigl(\mathbf{x}_{l}(s;y)\bigr),
            \nabla V\!\bigl(\mathbf{x}_{l}(s;y)\bigr)
          \bigr\rangle}_{\le 0}
        \,ds
  \;<\;c,
\]
a contradiction.  Therefore $\Omega_c$ is forward-invariant.

 \medskip
    \noindent Since the closure of the forward-invariant set is compact, the trajectory
\(\mathbf{x}_{l}(t;x)\) cannot blow up in finite time.  Because
\(l \in \xCone\!\bigl(\cl{\Omega_c},\mathbb{R}^{N}\bigr)\),
the map \(l\) is Lipschitz continuous on \(\cl{\Omega_c}\);
therefore the solution exists for all \(t \ge 0\)
(see the proof of~\cite[Theorem 3.3]{Khalil2002}).
    
    \medskip
    \noindent Regarding the asymptotic stability,  assume that there exists an initial state $x\in\Omega_c$ such that
$
  \lVert\mathbf{x}_{l}(t;x)\rVert$ does not converge to zero for $t\to\infty$, i.e., there is a radius $\rho>0$ with the property that for every $T>0$
there exists $t>T$ satisfying $\lVert\mathbf{x}_{l}(t;x)\rVert>\rho$.
Hence we can choose an increasing sequence
$0<t_{1}<t_{2}<\dots$ such that
$\lVert\mathbf{x}_{l}(t_{n};x)\rVert>\rho$ for all $n\in\mathbb{N}$.
Because the vector field $l$ is continuous in the compact set
$\cl{\Omega_c}$, it is bounded; therefore, by the mean value theorem and the forward invariance of $\Omega_c$, every trajectory
$t\mapsto\mathbf{x}_{l}(t;x)$ is uniformly continuous.
Consequently, there exists $\delta>0$ such that
\[
  \bigl\lVert\mathbf{x}_{l}(t;x)\bigr\rVert>\frac{\rho}{2}
  \qquad
  \forall\,t\in\bigl[t_{n}-\tfrac{\delta}{2},\,t_{n}+\tfrac{\delta}{2}\bigr],
  \ \ n\ge 2.
\]
Define
\[
  M := 
  \max_{x\in\operatorname{cl} \left(\Omega_c\setminus B_{\rho}(0)\right)}
      \bigl\langle l(x),\nabla V(x)\bigr\rangle
  < 0,
\]
which is well-defined because
$\langle l(x),\nabla V(x)\rangle$ is continuous and strictly negative
on $\cl{\Omega_c}\setminus\{0\}$.
Using the forward invariance of $\Omega_c$ and integrating the
 derivative of $V$ along the trajectory yields
\[
\begin{aligned}
  V\!\bigl(\mathbf{x}_{l}(t_{n}+\tfrac{\delta}{2};x)\bigr)
  &=
  V(x)
  +\int_{0}^{t_{n}+\delta/2}
      \bigl\langle
        l\!\bigl(\mathbf{x}_{l}(s;x)\bigr),
        \nabla V\!\bigl(\mathbf{x}_{l}(s;x)\bigr)
      \bigr\rangle
      ds \\[6pt]
  &\le
  c
  +\sum_{i=2}^{n}
     \int_{t_{i}-\delta/2}^{t_{i}+\delta/2}
       \bigl\langle
         l\!\bigl(\mathbf{x}_{l}(s;x)\bigr),
         \nabla V\!\bigl(\mathbf{x}_{l}(s;x)\bigr)
       \bigr\rangle
       ds \\[6pt]
  &\le
  c + (n-2)\,M\,\frac{\delta}{2}.
\end{aligned}
\]
Since $M<0$, the right-hand side is negative for sufficiently large
$n$, contradicting the non-negativity of $V$.
Therefore
\(
  \displaystyle\lim_{t\to\infty}
  \lVert\mathbf{x}_{l}(t;x)\rVert
  =0.\) Moreover, since $V$ is continuous, we have
  \begin{align*}
      c> V(\mathbf{x}_{l}(t;x)) \Longrightarrow  c> V(0) = 0 \Longrightarrow 0\in \Omega_c.
  \end{align*}

\end{proof}

%%-----------------------------
\end{document}